\documentclass{elsarticle}
\usepackage{etoolbox}
\usepackage{amssymb}
\usepackage{amsmath}
\usepackage{amsthm}
\usepackage{mathtools}
\usepackage{subcaption}
\usepackage{enumitem}
\usepackage{csquotes}
\usepackage{xcolor}
\usepackage{hyperref}
\usepackage{tikz-cd}
\usepackage[shortcuts]{extdash}
\usepackage{macros}
\usepackage{todonotes}
\usepackage{bbm}

\theoremstyle{plain}
\newtheorem{theorem}{Theorem}[section]
\newtheorem{corollary}[theorem]{Corollary}
\newtheorem{lemma}[theorem]{Lemma}
\newtheorem{proposition}[theorem]{Proposition}

\theoremstyle{definition}
\newtheorem{definition}[theorem]{Definition}

\newtheorem{example}[theorem]{Example}

\theoremstyle{remark}
\newtheorem{remark}[theorem]{Remark}

\begin{document}

\begin{frontmatter}

\title{Computads with invertible generators\texorpdfstring{\\}{}
    for weak \texorpdfstring{\(\omega\)}{ω}-categories}

\author{Thibaut Benjamin}
\ead{thibaut.benjamin@universite-paris-saclay.fr}

\affiliation{
    organisation={Laboratoire Méthodes Formelles, CNRS, ENS Paris-Saclay,
        Université Paris-Saclay},
    addressline={4 avenue des Sciences},
    city={Gif-sur-Yvette},
    postcode={91190},
    country={France}
}

\author{Camil Champin}
\ead{camil.champin@ens-lyon.com}

\affiliation{
    organisation={Ecole Normale Superieure de Lyon},
    addressline={15 parvis René Descartes},
    city={Lyon},
    postcode={69342},
    country={France}
}

\author{Ioannis Markakis}
\ead{ioannis.markakis@cl.cam.ac.uk}

\affiliation{
    organisation={Department of Computer Science and Technology,
        University of Cambridge},
    addressline={William Gates Building, 15 JJ Thompson Avenue},
    city={Cambridge},
    postcode={CB3 0FD},
    country={UK}
}

\begin{abstract}
    We extend the notion of computads for weak \(\omega\)\=/categories to allow
    marking certain generators as invertible, and describe inductively the free
    \(\omega\)\=/categories they generate. This gives a simple, finite
    description of the walking equivalences, the \(\omega\)\=/categories
    classifying invertible cells. We then construct a coreflection from
    generalised to ordinary computads, preserving the generated
    \(\omega\)\=/categories, and conclude that \(\omega\)\=/categories
    generated by generalised computads are cofibrant. Finally, we study
    the subcategory of generalised computads and generator-preserving morphisms,
    and show that it is a presheaf topos, similarly to the case of ordinary
    computads.
\end{abstract}


\end{frontmatter}


\section{Introduction}
\label{sec:intro}




Isomorphisms play a crucial role in category theory, providing the right notion
of equivalence between objects in a category and between categories. The same
role is played by invertible cells in \(n\)\=/categories, which are defined
inductively as cells that admit inverses up to an equivalence in a hom
\((n-1)\)\=/category. In infinite-dimensional categories, invertibility is a
richer concept with several variants, as explained by Ozornova, Rovelli and
Walde~\cite{ozornovaCoresLocalizations$inftyinfty$categories2026}, or by
Henry and Loubaton~\cite{loubatonInductiveModelStructure2023}.
In this article,
we focus on coinductively invertible cells in weak \(\omega\)\=/categories in
the sense of Batanin~\cite{bataninMonoidalGlobularCategories1998} and
Leinster~\cite{leinsterHigherOperadsHigher2003}. As the name suggests, these are
defined coinductively as cells that admit inverses up to a higher coinductively
invertible cells.

We propose a generalisation of computads for weak \(\omega\)\=/categories, in
which certain generators are marked as invertible. We do this following
Dean et al.~\cite{deanComputadsWeak$omega$categories2024}, who first
define mutually inductively computads and the underlying globular set of the
\(\omega\)\=/categories they generate. They then define \(\omega\)\=/categories
as algebras for the induced monad on globular sets and show that this monad is
equivalent to the free weak \(\omega\)\=/category monad of
Leinster~\cite{leinsterHigherOperadsHigher2003}. We extend their inductive
construction to define computads with invertible generators by adding new
operations forcing invertible generators to be coinductively invertible. This
extension induces the same monad on globular sets, so our computads
still give rise to weak \(\omega\)\=/categories in the sense of Leinster.

Instead of a two-sided inverse, our framework coinductively produces separate
left and right inverses for each invertible generator. This allows us to give a
concise description of the walking equivalence of Ozornova and Rovelli~\cite{
ozornova$inftyn$categoriesContext2025} as the computad with two objects
and an invertible generator between them, and show that it is the
\(\omega\)\=/category classifying invertible arrows: morphisms out of the
walking equivalence correspond to a morphism together with the data witnessing
its invertibility. This computad and its suspensions, which admit a similar
description, are conjectured to play a fundamental role in the homotopy theory
of weak \(\omega\)\=/categories; Fujii, Hoshino and
Maehara~\cite{fujii$o$equifibrationsStrictWeak2025} have shown that
\(\omega\)\=/equifibrations, the \(\omega\)\=/categorical analogue of
isofibrations, are precisely the maps with the right lifting property against
the suspensions of the source inclusion of the walking equivalence.

Our main result shows that the additional expressive power of computads with
invertible generators comes at no cost: we construct a coreflection from the
category of computads with invertible generators to that of ordinary computads,
and show that it preserves the generated \(\omega\)\=/categories. The
correflection turns invertible generator, their inverses and the higher
cancellator cells into ordinary generators yielding, in general, a much larger
presentation of the same \(\omega\)\=/category. Combined with previous work
by Dean et al.~\cite{deanComputadsWeak$omega$categories2024} and
Markakis~\cite{markakisComputadsGeneralisedSignatures2024}, this implies that
\(\omega\)\=/categories generated by computads with invertible generators are
cofibrant for the factorisation system cofibrantly generated by the
inclusions of spheres into disks.

We conclude by studying the subcategory of computads with rigid morphisms,
that is, morphisms sending generators to generators. As with ordinary
computads, we show that this category is a presheaf topos and that the
restriction of the underlying globular set functor to this subcategory is
familially representable. This is in contrast to the case of strict
\(\omega\)\=/categories, where the categories of computads and rigid morphisms
fail to be presheaf toposes above dimension 2, as shown by
Makkai and Zawadowski~\cite{makkaiCategory3computadsNot2008}, and further
explained by Cheng~\cite{chengDirectProofThat2012}.

\subsection*{Related Work}
\label{subsec:intro-related-work}

The weak \(\omega\)\=/categories discussed in this article are originally due to
Leinster~\cite{leinsterHigherOperadsHigher2003}, and they are a variant of these
by Batanin~\cite{bataninMonoidalGlobularCategories1998}. More precisely, we
use the inductive definition of Dean et al.~\cite{
deanComputadsWeak$omega$categories2024} which is inspired by the type-theoretic
one of Finster and Mimram~\cite{finsterTypetheoreticalDefinitionWeak2017}.
Equivalences between these definitions and the ones proposed by
Maltsiniotis~\cite{maltsiniotisGrothendieck$infty$groupoidsStill2010} have been
established by Ara~\cite{ara$infty$groupoidesGrothendieckVariante2010} and
Bourke~\cite{bourkeIteratedAlgebraicInjectivity2020}, by Benjamin, Finster and
Mimram~\cite{benjaminGlobularWeak$omega$categories2024} and by Benjamin,
Markakis and Sarti~\cite{benjaminCaTTContextsAre2024}.

Coinductive invertibility in weak \(\omega\)\=/categories was originally
introduced by Cheng~\cite{cheng$omega$categoryAllDuals2007}, and shown
equivalent to the version used in this article with separate left and right
inverses by Rice~\cite{riceCoinductiveInvertibilityHigher2020}. This version
of invertibility is the one used by Ozornova and Rovelli~\cite{
ozornovaWhatEquivalenceHigher2024} to describe the
walking equivalence for \(\omega\)\=/categories, which was then shown to be
contractible in the strict case by Hadzihasanovic et
al.~\cite{hadzihasanovicModelCoherentWalking2024}. The same walking equivalence
also appears in the work of Fujii, Hoshino and
Maehara~\cite{fujii$o$equifibrationsStrictWeak2025} characterising
\(\omega\)\=/equifibrations.



\subsection{Overview of the paper}
\label{subsec:intro-overview}

In Section~\ref{sec:pasting-diagrams}, we recall the definition of globular
pasting diagrams. Section~\ref{sec:definition} contains the main definition of
the article, that of computads with invertible generators.
In Section~\ref{sec:omega-cats} we compare computads with invertible generators
to ordinary computads, and show that the free \(\omega\)\=/categories they
generate can be described as colimits of pushouts of inclusions of spheres into
disks and walking equivalences. We close with Section~\ref{subsec:rigid} where
we study the subcategory of rigid morphisms and show that it is a presheaf
topos.



\section{Globular pasting diagrams}
\label{sec:pasting-diagrams}

\noindent
We start by introducing the category of globes and the collection of pasting
diagrams. These are the underlying category of shapes, and the arities of the
operations of weak \(\omega\)\=/categories respectively, where \(\omega\) is
the least infinite ordinal. The weak \(\omega\)\=/categories we consider in this
paper are those of Batanin~\cite{bataninMonoidalGlobularCategories1998} and
Leinster~\cite{leinsterHigherOperadsHigher2003}, which have been shown to be a
equivalent to these of Grothendieck and
Maltsiniotis~\cite{maltsiniotisGrothendieck$infty$groupoidsStill2010} by
Ara~\cite{ara$infty$groupoidesGrothendieckVariante2010} and
Bourke~\cite{bourkeIteratedAlgebraicInjectivity2020}.

\subsection{Globular sets}

\noindent
We first define the categories of globes and globular sets, as well as the
suspension and the wedge sum of globular sets. These operations allow us
to give a concise description of pasting diagrams below. We also introduce
truncated variants of these categories and operations that allow us to reason
recursively on the dimension.

\begin{definition}\label{def:globe-and-gSet}
    The \emph{globe category} \(\Globe\) has objects natural
    numbers, and morphisms generated by the \emph{source} and \emph{target} maps
    \(\cosrc_n,\cotgt_n \colon n \to (n+1)\) for every \(n\in\Nat\) under the
    globularity relations
    \begin{align}\label{eq:coglobularity}
        \cosrc_{n+1}\cosrc_n &= \cotgt_{n+1}\cosrc_n &
        \cosrc_{n+1}\cotgt_n &= \cotgt_{n+1}\cotgt_n
    \end{align}
    The category \(\gSet\) of \emph{globular sets} is the category of presheaves
    on \(\Globe\).
\end{definition}

\noindent
We denote by \(X_n\) the set \(X(n)\) for a globular set \(X\) and call its
elements the \emph{\(n\)\=/cells of \(X\)}. We denote also by \(\src_n,\tgt_n
\colon X_{n+1}\to X_n\) the functions \(X(\cosrc_n)\) and \(X(\cotgt_n)\)
respectively and call them the \emph{source} and \emph{target} functions. More
generally, we denote for \(m>n\) by \(\src^m_n, \tgt^m_n \colon X_m \to X_n\)
the iterated composite of the source and target maps, dropping the superscript
\(m\) when it can be easily inferred.
\begin{align}\label{eq:globularity}
    \begin{tikzcd}[ampersand replacement = \&]
        {X_0} \& {X_1} \& {X_2} \& {\cdots}
        \arrow["{\src_0}"', shift right=2, from=1-2, to=1-1]
        \arrow["{\tgt_0}", shift left=2, from=1-2, to=1-1]
        \arrow["{\tgt_1}", shift left=2, from=1-3, to=1-2]
        \arrow["{\src_1}"', shift right=2, from=1-3, to=1-2]
        \arrow["{\tgt_2}", shift left=2, from=1-4, to=1-3]
        \arrow["{\src_2}"', shift right=2, from=1-4, to=1-3]
    \end{tikzcd} &&
    \begin{aligned}
        \src_n\src_{n+1} &= \src_n\tgt_{n+1} \\
        \tgt_n\src_{n+1} &= \tgt_n\tgt_{n+1}
    \end{aligned}
\end{align}
We say furthermore that a pair of \(n\)\=/cells are \emph{parallel} when
\(n = 0\) or they have the same source and target, and we denote the set of
parallel \(n\)\=/cells by \(\Par_nX\).

\begin{definition}\label{def:globes-spheres}
    The \emph{\(n\)\=/globe} \(\Disk^n\) is the representable globular set
    \(\Globe(-,n)\). The \emph{\((n\!-\!1)\)\=/sphere} \(\iota_n\colon
    \Sphere^{n-1}\hookrightarrow \Disk^n\) is the globular subset obtained by
    removing the unique \(n\)\=/cell of \(\Disk^n\).
\end{definition}

\noindent
Equivalently, the \(n\)\=/sphere with the inclusion can be defined via the
following pushout diagram
\begin{equation}\label{eq:sphere-as-po}
    \begin{tikzcd}
        {\Sphere^{n-1}} & {\Disk^n} \\
        {\Disk^n} & {\Sphere^n} \\
        && {\Disk^{n+1}}
        \arrow["{\sphereInc_{n}}", from=1-1, to=1-2]
        \arrow["{\sphereInc_{n}}"', from=1-1, to=2-1]
        \arrow["{\inc_\sigma}", dashed, from=1-2, to=2-2]
        \arrow["{\inc_\tau}"', dashed, from=2-1, to=2-2]
        \arrow["\lrcorner"{anchor=center, pos=0.125, rotate=180},
            draw=none, from=2-2, to=1-1]
        \arrow["{\cosrc_{n}}", bend left, from=1-2, to=3-3]
        \arrow["{\cotgt_{n}}"', bend right, from=2-1, to=3-3]
        \arrow["{\sphereInc_{n+1}}"{description}, dashed, from=2-2, to=3-3]
    \end{tikzcd}
\end{equation}
starting from \(\Sphere^{-1} = \emptyset\) being the initial globular set. By
the Yoneda lemma and continuity of hom functors, we get natural isomorphisms
\begin{align}\label{eq:disk-sphere-representable}
    \gSet(\Disk^n, X) &\cong X_n & \gSet(\Sphere^n, X) &\cong \Par_n(X)
\end{align}
such that composition with the inclusion
\(\iota_n\colon\Sphere^{n-1}\to\Disk^n\) sends an \((n\!+\!1)\)\=/cell
to its source and target.

\begin{definition}\label{def:suspension}
    The suspension functor \(\Susp\colon\gSet\to\gSet\) sends a globular set
    \(X\) to the globular set consisting of
    \begin{align}\label{eq:suspension}
        (\Susp X)_0 &= \set{\sigma,\tau} & (\Susp X)_{n+1} &= X_n
    \end{align}
    with \(\sigma\) and \(\tau\) the source and target of every \(1\)\=/cell
    respectively.
    \begin{equation}\label{eq:suspension-image}
        \begin{tikzcd}[ampersand replacement = \&]
            {\set{\sigma,\tau}} \& {X_0} \& {X_1} \& {X_2} \& {\cdots}
            \arrow["{\sigma}"', shift right=2, from=1-2, to=1-1]
            \arrow["{\tau}", shift left=2, from=1-2, to=1-1]
            \arrow["{\src_0}"', shift right=2, from=1-3, to=1-2]
            \arrow["{\tgt_0}", shift left=2, from=1-3, to=1-2]
            \arrow["{\tgt_1}", shift left=2, from=1-4, to=1-3]
            \arrow["{\src_1}"', shift right=2, from=1-4, to=1-3]
            \arrow["{\tgt_2}", shift left=2, from=1-5, to=1-4]
            \arrow["{\src_2}"', shift right=2, from=1-5, to=1-4]
        \end{tikzcd}
    \end{equation}
\end{definition}

\noindent
It is easy to see that the suspension functor preserves the globes and spheres
in the sense that there exist isomorphisms
\begin{align}\label{eq:suspension-disk-sphere}
    \Susp \Disk^n &\cong \Disk^{n+1} & \Susp \Sphere^{n-1} &\cong \Sphere^{n}
\end{align}
for every \(n\in \Nat\) compatible with the source and target morphisms and the
inclusion of spheres into globes. Moreover, the suspension functor is a
\emph{parametric left adjoint} in the sense that it factors as a left adjoint to
the slice of \(\gSet\) under \(\Susp \emptyset\cong \Sphere^0\cong
\Disk^0+\Disk^0\) followed by the forgetful functor from the slice:
\[\begin{tikzcd}
	\gSet & {\Sphere^0\downarrow\gSet} & \gSet
	\arrow[""{name=0, anchor=center, inner sep=0},
        "\Susp", shift left=2, from=1-1, to=1-2]
	\arrow[""{name=1, anchor=center, inner sep=0},
        dashed, "\Hom", shift left=2, from=1-2, to=1-1]
	\arrow["\pr", from=1-2, to=1-3]
	\arrow["\dashv"{anchor=center, rotate=-90}, draw=none, from=0, to=1]
\end{tikzcd}\]
By the isomorphisms above, the slice \(\Sphere^0\downarrow\gSet\) is isomorphic
to the category of globular sets \(X\) equipped with two chosen \(0\)\=/cells
\(a,b\). Under this isomorphism, the right adjoint is the functor defined by
\begin{equation}\label{eq:hom}
    \Hom_X(a,b)_n = \set{x\in X_{n+1}\such\src_0(x)=a\text{ and }\tgt_0(x)=b}
\end{equation}
with source and target functions those of \(X\). The existence of this right
adjoint implies that the suspension functor \(\Susp\colon\gSet\to\gSet\)
preserves connected colimits, since the forgetful functor creates them.

\begin{definition}\label{def:wedge-sum}
    The \emph{wedge sum} \(X\vee Y\) of globular sets \(X\) and \(Y\) with
    chosen \(0\)\=/cells \(a,b\in X_0\) and \(c,d\in Y_0\) respectively is
    given by the following pushout square of globular sets:
    \begin{equation}\label{eq:wedge-sum}
        \begin{tikzcd}
            && {X\vee Y} \\
            & X && Y \\
            {\Disk^0} && {\Disk^0} && {\Disk^0}
            \arrow["\lrcorner"{anchor=center, pos=0.125, rotate=-45}, draw=none,
                from=1-3, to=3-3]
            \arrow[dashed, from=2-2, to=1-3]
            \arrow[dashed, from=2-4, to=1-3]
            \arrow["a", from=3-1, to=2-2]
            \arrow["b"', from=3-3, to=2-2]
            \arrow["c", from=3-3, to=2-4]
            \arrow["d"', from=3-5, to=2-4]
        \end{tikzcd}
    \end{equation}
    Letting the image of \(a\) and \(d\) in the pushout be its chosen
    \(0\)\=/cells, we get a monoidal structure on the slice category
    \(\Sphere^0\downarrow \gSet\) with unit object the unique morphism
    \(\Sphere^0\to \Disk^0\)
\end{definition}

\noindent
One way to see that \((\Sphere^0\downarrow \gSet,\vee, \Disk^0)\) is a monoidal
structure is to observe that the slice category is equivalently the category of
endocospans of globular sets on \(\Disk^0\) and that the wedge sum is precisely
the composition of cospans.

\begin{definition}\label{def:trunc-globe-and-gset}
    The category \(\Globe_n\) for \(n\in \Nat\cup\set{-1,\omega}\) is the full
    subcategory of \(\Globe\) with objects natural numbers \(k \le n\). The
    category \(\gSet_n\) of \(n\)\=/globular sets is the category of presheaves
    on \(\Globe_n\). The subcategory inclusion
    \(\Globe_n\hookrightarrow\Globe_m\) for \(m \ge n\) gives rise to an adjoint
    triple of functors, satisfying the following equations:
    \begin{align}\label{eq:sk-tr-cosk}
        \begin{tikzcd}[ampersand replacement = \&]
            {\gSet_m} \& {\gSet_n}
            \arrow[""{name=0, anchor=center, inner sep=0},
                "{\tr^m_n}"{description}, from=1-1, to=1-2]
            \arrow[""{name=1, anchor=center, inner sep=0},
                "{\cosk^m_n}", bend left=40, from=1-2, to=1-1]
            \arrow[""{name=2, anchor=center, inner sep=0},
                "{\sk^m_n}"', bend right=40, from=1-2, to=1-1]
            \arrow["\dashv"{anchor=center, rotate=-90},
                draw=none, from=0, to=1]
            \arrow["\dashv"{anchor=center, rotate=-90},
                draw=none, from=2, to=0]
        \end{tikzcd} &&
        \begin{aligned}
            \tr^n_m \sk^n_m &= \id \\
            \tr^n_m \cosk^n_m &= \id
        \end{aligned}
    \end{align}
    The \emph{skeleton}, \emph{truncation}, and \emph{coskeleton} functors are
    defined by left Kan extension, restriction, and right Kan
    extension along the inclusion respectively, and the equations follow by the
    inclusion being fully faithful.
\end{definition}

\noindent
This adjoint triple can be described quite explicitly as well. The truncation
functor forgets the values of an \(m\)\=/globular set \(X\) for
\(m<k\le n\). The skeleton functor and coskeleton functors extend an
\(n\)\=/globular set \(Y\) by the empty set and its set of \(n\)\=/parallel
pairs respectively:
\begin{align}\label{eq:sk-cosk-definition}
    (\sk^m_n X)_k &= \begin{cases}
        X_k&\text{for }k\le n\\
        \emptyset&\text{otherwise}
    \end{cases}&
    (\cosk^m_n X)_k &= \begin{cases}
        X_k&\text{for }k\le n\\
        \Par_n(X)&\text{otherwise}
    \end{cases}
\end{align}
Equations~\eqref{eq:sk-tr-cosk} implies that the skeleton and coskeleton
functors are also fully faithful, and we will them use for \(m = \omega\) to
implicitly identify \(n\)\=/globular sets with globular sets that have no cells
above dimension \(n\). Under this identification, we see that the \(n\)\=/globe
and the \(n\)\=/sphere are also \(n\)\=/globular sets, and that the suspension
and wedge sum restrict to operations
\begin{align}\label{eq:truncated-supsension-wedge}
    \Susp&\colon\gSet_n\to \gSet_{1+n} &
    \vee&\colon(\Sphere^0\downarrow \gSet_n)\times(\Sphere^0\downarrow \gSet_n)
        \to \Sphere^0\downarrow \gSet_n
\end{align}
for every \(n\in \Nat\cup\set{-1,\omega}\) where we let \(1+\omega = \omega\).

\subsection{Pasting diagrams}

\noindent
Pasting diagrams are a family of globular sets that plays an important role in
many approaches to higher categories, including the weak \(\omega\)\=/categories
we study in this paper and Rezk's
\(\Theta_n\)\=/spaces~\cite{rezkCartesianPresentationWeak2010}. These are the
globular sets that famillialy represent the free strict \(\omega\)\=/category
monad, and they have been described in terms of rooted, planar trees by
Batanin~\cite{bataninMonoidalGlobularCategories1998}, in terms of lists by
Leinster~\cite{leinsterHigherOperadsHigher2003}, and as certain colimits of
globes by Maltsiniotis~\cite{maltsiniotisGrothendieck$infty$groupoidsStill2010}
and Ara~\cite{ara$infty$groupoidesGrothendieckVariante2010}. In this paper, we
will take the approach of Dean et
al~\cite{deanComputadsWeak$omega$categories2024} and describe them as
inductive structures.

\begin{definition}\label{def:trees}
    The set \(\Tree\) of \emph{Batanin trees} is defined \emph{inductively} by
    the existence of a tree \(\br L\) for every list of Batanin trees
    \(L = [B_1,\dots,B_n]\).
\end{definition}

\noindent
What we mean by a structure being defined inductively is that it is the carrier
of the initial algebra for some polynomial endofunctor defined in terms of some
\emph{constructors}. In this particular example, we have a single constructor
\begin{align}\label{eq:br-trees}
    \br&\colon \List \Tree \to \Tree & \List X = \coprod_{n\in \Nat} X^n
\end{align}
taking a list of trees \([B_1,\dots,B_n]\) to the rooted tree with branches the
trees \(B_i\), so \(\Tree\) is the carrier of the initial algebra for the free
monoid endofunctor \(\List\) on \(\Set\), and \(\br\) is its structure morphism.
This inductive description allows us to define operations on Batanin trees
recursively:

\begin{definition}\label{def:dim-bdry-tree}
    The \emph{dimension} of a Batanin tree is the natural number defined by the
    recursive formula
    \begin{equation}\label{eq:dim-tree}
        \dim(\br[B_1,\dots,B_n]) = \sup\set{\dim(B_1) + 1, \dots, \dim(B_n) + 1}
    \end{equation}
    The \emph{\(k\)\=/boundary} of a Batanin tree for \(k\in \Nat\) is the
    tree defined recursively by
    \begin{equation}\label{eq:bdry-tree}
        \begin{aligned}
        \partial_0(\br[B_1,\dots,B_n])
            &= \br[] \\
        \partial_{k+1}(\br[B_1,\dots,B_n])
            &= \br[\partial_k B_1,\dots,\partial_k B_n]
        \end{aligned}
    \end{equation}
\end{definition}

\noindent
The dimension of a Batanin tree is the height of the corresponding rooted planar
tree, that is the maximum of the distances from the root to the leaves. The
\(k\)\=/boundary operation acts on a tree by removing all nodes whose distance
from the root is greater than \(k\), as illustrated in
Figure~\ref{fig:tree-bdry}.

\begin{figure}
    \centering
    \begin{tabular}{|c|ccc|}
        \hline
        & List & Planar rooted tree & positions \\
        \hline
        \(B\) &
        \(\br[\br[\br[],\br[]],\br[]]\) &
        \(
        \begin{tikzcd}[column sep = tiny, row sep = small]
            \bullet && \bullet \\
            & \bullet && \bullet \\
            && \bullet
            \arrow[no head, from=2-2, to=1-1]
            \arrow[no head, from=2-2, to=1-3]
            \arrow[no head, from=3-3, to=2-2]
            \arrow[no head, from=3-3, to=2-4]
        \end{tikzcd}
        \) &
        \(
        \begin{tikzcd}[column sep=small]
            \bullet && \bullet && \bullet
            \arrow[""{name=0, anchor=center, inner sep=0}, from=1-1, to=1-3]
            \arrow[""{name=1, anchor=center, inner sep=0}, bend left = 60,
                from=1-1, to=1-3]
            \arrow[""{name=2, anchor=center, inner sep=0}, bend right = 60,
                from=1-1, to=1-3]
            \arrow[from=1-3, to=1-5]
            \arrow[shorten <=3pt, shorten >=3pt, Rightarrow, from=1, to=0]
            \arrow[shorten <=3pt, shorten >=3pt, Rightarrow, from=0, to=2]
          \end{tikzcd}
        \)
        \\
        \hline
        \(\partial_1 B\) &
        \( \br[\br[], \br[]] \) &
        \(
          \begin{tikzcd}[column sep = tiny, row sep = small]
            \bullet && \bullet \\
            & \bullet
            \arrow[no head, from=2-2, to=1-1]
            \arrow[no head, from=2-2, to=1-3]
          \end{tikzcd}
        \) &
        \(
          \begin{tikzcd}[column sep=small]
            \bullet && \bullet && \bullet
            \arrow[from=1-1, to=1-3]
            \arrow[from=1-3, to=1-5]
          \end{tikzcd}
        \)
        \\
        \hline
    \end{tabular}
    \caption{A Batanin tree and its boundary}\label{fig:tree-bdry}
\end{figure}

\begin{definition}\label{def:pos-src-tgt-inclusions}
    The \emph{pasting diagram} associated to a Batanin tree \(B\) is the
    globular set defined recursively by letting
    \begin{equation}
        \Pos (\br[B_1,\dots,B_n])
            = (\Susp\Pos B_1) \vee \dots \vee (\Susp\Pos B_n)
    \end{equation}
    The \emph{source} and \emph{target} maps
    \(\sigma_k^B,\tau_k^B\colon \Pos(\partial_k B)\to \Pos(B)\) are defined
    recursively by letting \(\sigma_0^B\) and \(\tau_0^B\) the morphisms out
    of \(\Disk^0\) corresponding to the first and second chosen \(0\)\=/cells
    of \(\Pos(B)\) respectively, and:
    \begin{equation}\label{eq:src-tgt-inclusions}
        \begin{aligned}
            \sigma_{k+1}^{\br[B_1,\dots,B_n]} &= (\Susp \sigma_k^{B_1}) \vee
                \dots \vee (\Susp \sigma_k^{B_n}) \\
            \tau_{k+1}^{\br[B_1,\dots,B_n]} &= (\Susp \tau_k^{B_1}) \vee
                \dots \vee (\Susp \tau_k^{B_n})
        \end{aligned}
    \end{equation}
\end{definition}

\noindent
Using observation~\eqref{eq:truncated-supsension-wedge} and recursion on the
Batanin tree \(B\), we can see that the pasting diagram \(\Pos(B)\) has a finite
number of cells, all of which have dimension at most \(\dim(B)\). Additionally,
we can check that the source and target maps satisfy analogues of the
globularity equations~\eqref{eq:coglobularity}, as shown for instance by Dean et
al.~\cite[Proposition~2.8]{deanComputadsWeak$omega$categories2024}. More
precisely, we have for \(k<m\) that \(\partial_k\partial_m B =\partial_k B\) and
\begin{align}
    \sigma_m^B\circ\sigma_k^{\partial_m B}
        &= \tau_m^B\circ\sigma_k^{\partial_m B} &
    \sigma_m^B\circ\tau_k^{\partial_m B}
        &= \tau_m^B\circ\tau_k^{\partial_m B}
\end{align}
and that \(\partial_k B = B\) whenever \(k\ge \dim(B)\) with
\(s_k^B = t_k^B = \id_{\Pos(B)}\) in that case.

\begin{definition}\label{def:linear-trees}
    The \emph{linear trees} are defined recursively on \(n\in\Nat\) by:
    \begin{align}\label{eq:def-linear-trees}
        D_0 &= \br[] & D_{n+1} &= \br[D_n]
    \end{align}
\end{definition}

\noindent
Unwrapping the definition, we can see that \(D_n\) is a Batanin tree of
dimension \(n\) and that its boundaries are given by
\(\partial_kD_{n} = D_k\) when \(k<n\). Using the natural isomorphisms
in~\eqref{eq:suspension-disk-sphere}, we can show that the pasting diagram
\(\Pos(D_n)\) is isomorphic to the \(n\)\=/globe \(\Disk^n\). Moreover,
these isomorphisms can be chosen natural, in that the following squares
commute:
\begin{align}\label{eq:linear-trees-src-tgt}
    \begin{tikzcd}[ampersand replacement = \&]
        {\Pos(D_n)} \& {\Disk^n} \\
        {\Pos(D_{n+1})} \& {\Disk^{n+1}}
        \arrow["\sim", from=1-1, to=1-2]
        \arrow["{\sigma_n^{D_{n+1}}}"', from=1-1, to=2-1]
        \arrow["{\sigma_n}", from=1-2, to=2-2]
        \arrow["\sim", from=2-1, to=2-2]
    \end{tikzcd} &&
    \begin{tikzcd}[ampersand replacement = \&]
        {\Pos(D_n)} \& {\Disk^n} \\
        {\Pos(D_{n+1})} \& {\Disk^{n+1}}
        \arrow["\sim", from=1-1, to=1-2]
        \arrow["{\tau_n^{D_{n+1}}}"', from=1-1, to=2-1]
        \arrow["{\tau_n}", from=1-2, to=2-2]
        \arrow["\sim", from=2-1, to=2-2]
    \end{tikzcd}
\end{align}

\begin{definition}\label{def:binary-trees}
    We define a family of trees \(P^k_n\) for every pair of natural
    numbers \(n > k\) recursively by the equations:
    \begin{align}\label{eq:binary-trees-def}
        P^0_{n+1} &= \br[D_n, D_n] &
        P^{k+1}_{n+1} &= \br[P^k_n]
    \end{align}
\end{definition}

\noindent
A recursive argument shows that \(P^k_n\) is a Batanin tree of
dimension \(n\) whose boundaries are given by
\begin{equation}\label{eq:binary-trees-bdry}
    \partial_m P^k_n = \begin{cases}
        D_m&\text{for }m\le k \\
        P^k_{m}&\text{for }k < m < n \\
        P^k_n&\text{for }n\le m
    \end{cases}
\end{equation}
The pasting diagram \(\Pos(P_n^k)\) fits in a pushout square of the form:
\begin{equation}\label{eq:binary-trees-po}
    \begin{tikzcd}
        {\Disk^k} & {\Disk^n} \\
        {\Disk^n} & {\Pos(P^k_n)}
        \arrow["{\sigma_k^n}", from=1-1, to=1-2]
        \arrow["{\tau_k^n}"', from=1-1, to=2-1]
        \arrow[dashed, from=1-2, to=2-2]
        \arrow[dashed, from=2-1, to=2-2]
        \arrow["\lrcorner"{anchor=center, pos=0.125, rotate=180},
                draw=none, from=2-2, to=1-1]
    \end{tikzcd}
\end{equation}
When \(k = 0\), this is the pushout defining the wedge
sum~\eqref{eq:wedge-sum}. For \(k>0\), it follows by recursion using that
the suspension preserves globes and connected colimits. It follows that
\(P^k_n\) represents pairs of \emph{\(k\)\=/composable \(n\)\=/cells}:
\begin{equation}\label{eq:binary-trees-represent}
    \gSet(\Pos(P^k_n), X)\cong
    \set{(a,b)\in X_n\times X_n\such\tgt_ka = \src_k b}
\end{equation}
Under this isomorphism, composition with the \(m\)\=/source of \(P^k_n\)
corresponds to the function sending \((a,b)\) to \(\src_m a\) when
\(m\le k\) and to \((\src_ma,\src_mb)\) otherwise. Composition with the
\(m\)\=/target works dually. These pasting diagrams for small values of
\(n\) are shown in Figure~\ref{fig:binary}.

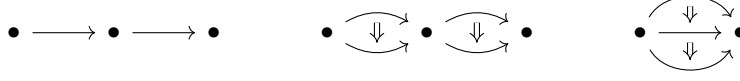
\begin{figure}
    \begin{align*}
        \begin{tikzcd}[ampersand replacement = \&]
            \bullet \& \bullet \& \bullet
            \arrow[from=1-1, to=1-2]
            \arrow[from=1-2, to=1-3]
        \end{tikzcd} &&
        \begin{tikzcd}[ampersand replacement=\&]
            \bullet \& \bullet \& \bullet
            \arrow[""{name=0, anchor=center, inner sep=0},
                bend left, from=1-1, to=1-2]
            \arrow[""{name=1, anchor=center, inner sep=0},
                bend right, from=1-1, to=1-2]
            \arrow[""{name=2, anchor=center, inner sep=0},
                bend left, from=1-2, to=1-3]
            \arrow[""{name=3, anchor=center, inner sep=0},
                bend right, from=1-2, to=1-3]
            \arrow[shorten <=3pt, shorten >=3pt, Rightarrow, from=0, to=1]
            \arrow[shorten <=3pt, shorten >=3pt, Rightarrow, from=2, to=3]
        \end{tikzcd} &&
        \begin{tikzcd}[ampersand replacement=\&]
            \bullet \& \bullet
            \arrow[""{name=0, anchor=center, inner sep=0},
                bend left = 60, from=1-1, to=1-2]
            \arrow[""{name=1, anchor=center, inner sep=0},
                bend right = 60, from=1-1, to=1-2]
            \arrow[""{name=2, anchor=center, inner sep=0}, from=1-1, to=1-2]
            \arrow[shorten <=3pt, shorten >=3pt, Rightarrow, from=0, to=2]
            \arrow[shorten <=3pt, shorten >=3pt, Rightarrow, from=2, to=1]
        \end{tikzcd}
    \end{align*}
\caption{The pasting diagrams \(P^1_0\), \(P^2_0\) and \(P^2_1\)}
\label{fig:binary}
\end{figure}


\section{Computads with invertible generators}
\label{sec:definition}

Computads, or polygraphs, were originally introduced in the context of
\(2\)\=/categories by Street~\cite{streetLimitsIndexedCategoryvalued1976}, and
then generalised to strict \(\omega\)\=/categories independently by
Street~\cite{streetAlgebraOrientedSimplexes1987} and
Burroni~\cite{burroniHigherdimensionalWordProblems1993}. For a recent exposition
of computads for strict \(\omega\)\=/categories, we refer to the book by Ara
et al.~\cite{araPolygraphsRewritingHigher2025}. The theory of computads was
further generalised by Batanin~\cite{bataninComputadsFinitaryMonads1998} to the
setting of weak \(\omega\)\=/categories and it was recently shown by Dean et
al.~\cite{deanComputadsWeak$omega$categories2024} that computads and the free
weak \(\omega\)\=/categories they generate can be described in an indexed
inductive-recursive manner.

In this section, we extend this approach by providing a similar indexed
inductive-recursive definition of \emph{computads with invertible generators}.
We proceed by ordinary recursion on the dimension
\(n\in\Nat\cup\set{-1,\omega}\) to define the following structures:
\begin{enumerate}[label=(\Alph*)]
    \item a category \(\iComp_n\) of \emph{\(n\)\=/computads with invertible
        generators},\label{item:iComp}
    \item a truncation functor \(\tr^n_{k}\colon \iComp_n\to \iComp_k\)
        for \(k<n\) removing the generators above dimension \(k\),
        \label{item:tr}
    \item a functor \(\iCell_n\colon\iComp_n\to\gSet_n\) returning the
        \emph{cells} of a computad, \label{item:iCell}
    \item a functor \(\iCptd_n\colon\gSet_n\to\iComp_n\) including globular sets
        to computads,\label{item:iCptd}
    \item a natural transformation
        \(\unitCell_n\colon\id\Rightarrow\iCell_n\iCptd_n\) picking the
        \emph{generators},\label{item:unitCell}
    \item a natural transformation
        \(\counitCell_n\colon\iCptd_n\iCell_n\Rightarrow\id\) evaluating cells,
        \label{item:counitCell}
    \item a subset \(\Full_n(B)\subseteq\Par_n(\iCell_n(\iCptd_n(\tr_n
        \Pos(B))))\) of \emph{full spheres} of a Batanin tree \(B\) for
        \(n<\omega\),\label{item:Full}
    \item a functor \(\iSet_n\colon\iComp_n\to\Set\) choosing some cells that
        will be made invertible for \(n<\omega\),
        \label{item:iSet}
    \item three natural transformations
        \(\fwd_n,\linv_n,\rinv_n\colon \iSet_n\Rightarrow\iCell_n(-)_n\)
        selecting
        an invertible \(n\)\=/cell, and its left and right inverse respectively
        for \(-1<n<\omega\),
        \label{item:fwd-linv-rinv}
    \item a natural \emph{binary composition} operation \(\compCell_k^n\) for
        \(0\le k<n<\omega\) defined on pairs of cells \((a,b)\in \iCell_n(C)\)
        such that \(\tgt_ka = \src_kb\),
        \label{item:bin-comp}
    \item a natural \emph{identity operation}
        \(\idCell_n\colon \iCell_n(C)_{n-1}\to
        \iCell_n(C)_n\) for \(0<n<\omega\)
        \label{item:idCell}
\end{enumerate}
Together with the definition of these structures, we will provide proofs for the
following properties:
\begin{enumerate}[label=(\roman*)]
    \item The truncation functors satisfy that \(\tr^m_k\tr^n_m = \tr^n_k\) for
        all \(k<m<n\).\label{item:prop-cocycle}
    \item The functor \(\iCptd_n\) is left adjoint to \(\iCell_n\) with unit
        \(\unitCell_n\) and counit \(\counitCell_n\), that is, the following
        triangle identities hold:\label{item:prop-adjunction}
        \begin{equation}
            \begin{aligned}\label{eq:trianglesCell}
                (\iCell_n\counitCell_n)\circ (\unitCell_n\iCell_n)
                    &= \id_{\iCell_n} \\
                (\counitCell_n\iCptd_n)\circ (\iCptd_n\unitCell_n)
                    &= \id_{\iCptd_n}
            \end{aligned}
        \end{equation}
    \item The adjunction commutes with the truncation functors, in that the
        following squares commute:
        \begin{align}\label{eq:truncation-cell}
            \begin{tikzcd}[ampersand replacement = \&]
                {\gSet_n} \& {\iComp_n} \\
                {\gSet_m} \& {\iComp_m}
                \arrow["{\iCell_n}", from=1-1, to=1-2]
                \arrow["{\tr^n_m}"', from=1-1, to=2-1]
                \arrow["{\tr^n_m}", from=1-2, to=2-2]
                \arrow["{\iCell_m}"', from=2-1, to=2-2]
            \end{tikzcd} &&
            \begin{tikzcd}[ampersand replacement = \&]
                {\iComp_n} \& {\gSet_n} \\
                {\iComp_m} \& {\gSet_m}
                \arrow["{\iCptd_n}", from=1-1, to=1-2]
                \arrow["{\tr^n_m}"', from=1-1, to=2-1]
                \arrow["{\tr^n_m}", from=1-2, to=2-2]
                \arrow["{\iCptd_m}"', from=2-1, to=2-2]
            \end{tikzcd}
        \end{align}
        and the following equalities between natural transformations hold:
        \begin{align}\label{eq:truncation-unit}
            \tr^n_m \unitCell_n &= \unitCell_m \tr^n_m &
            \tr^n_m \counitCell_n &= \counitCell_m \tr^n_m
        \end{align}
        \label{item:prop-tr-adjunction}
    \item The source and target of the natural transformations \(\fwd_n\),
        \(\linv_n\) and \(\rinv_n\) are related by
        \begin{equation}\label{eq:src-tgt-fwd-inv}
            \begin{aligned}
                \src_{n-1}\linv_n &=\src_{n-1}\rinv_n = \tgt_{n-1}\fwd_n \\
                \tgt_{n-1}\linv_n &=\tgt_{n-1}\rinv_n = \src_{n-1}\fwd_n \\
            \end{aligned}
        \end{equation}
        \label{item:prop-src-tgt-fwd-linv-rinv}
    \item The source and target of the composition operations are determined by:
        \begin{equation}\label{eq:src-tgt-id-comp}
            \begin{aligned}
                \src_{n-1} \idCell_n(c)
                    &= \tgt_{n-1} \idCell_n(c) = \id \\
                \src_{n-1}(a \compCell_k^n b)
                    &= \begin{cases}
                        \src_{n-1}(a), & \text{if }k=n-1 \\
                        \src_{n-1}(a) \compCell_k^n \src_{n-1}(b)
                            & \text{if }k<n-1 \\
                    \end{cases} \\
                \tgt_{n-1}(a \compCell_k^n b)
                    &= \begin{cases}
                        \tgt_{n-1}(b), & \text{if }k=n-1 \\
                        \tgt_{n-1}(a) \compCell_k^n \tgt_{n-1}(b)
                            & \text{if }k<n-1 \\
                    \end{cases}
            \end{aligned}
        \end{equation}
        \label{item:prop-src-tgt-comp-idcell}
    \item The following \(n\)\=/sphere of \(D_k\) for \(n \le k < \omega\)
        is well-defined and full: \label{item:prop-full}
        \begin{equation}\label{eq:idCell-bdry}
            A^{\id}_{n,k} = \begin{cases}
                (\gen_n(\id_n), \gen_n(\id_n)), &\text{if } k = n \\
                (\gen_n(\sigma_n), \gen_n(\tau_n)), &\text{if } k > n
            \end{cases}
        \end{equation}
        The following \(n\)\=/sphere of \(P^k_m\) for \(n\le m < \omega\) and
        \(k<m\) is well-defined and full:
        \begin{equation}\label{eq:compCell-bdry}
            A^{\comp}_{n,m,k} = \begin{cases}
                (\gen_n(\sigma_n^{P^k_m}), \gen_n(\tau_n^{P^k_m})),
                    &\text{if } n\le k \\
                (\iCell\iCptd(\sigma_n^{P^k_m})(c),
                \iCell\iCptd(\tau_n^{P^k_m})(c)),
                &\text{if } k < n
            \end{cases}
        \end{equation}
        for \(c = a \compCell^n_k b\) the composite of the two top-dimensional
        cells of \(P^k_n\).
\end{enumerate}

\subsection{Base cases}
\label{subsec:base-cases}

\noindent
We first describe the cases \(n = -1\) and \(n = 0\) explicitly. We start
by letting \({\iComp_{-1} = \gSet_{-1} = \star}\) be the terminal category,
\(\iCptd_{-1}\) and \(\iCell_{-1}\) be identity functors, and
\(\unitCell_{-1}\) and \(\counitCell_{-1}\) be identity natural transformations.
We let \(\Full_{-1}(B)\) contain the unique \((-1)\)\=/sphere of \(\Pos(B)\).
We let also \(\iSet_{-1}\) take the unique \((-1)\)\=/computad to the empty set.

We then define \({\iComp_{0} = \gSet_{0} \cong \Set}\) and let \(\tr^0_{-1}\)
the
unique functor in the terminal category. We let \(\iCptd_0\) and \(\iCell_0\) be
identity functors, and \(\unitCell_{0}\) and \(\counitCell_{0}\) be identity
natural transformations again. We let \((\sigma_0^B,\tau_0^B)\) be the unique
full \(0\)\=/sphere of \(\Pos(B)\) for every Batanin tree \(B\). We finally
define \(\iSet_0\) to be constant at the empty set, since only
positive-dimensional cells may be inverted, and we let \(\fwd_n\), \(\linv_n\)
and \(\rinv_n\) the unique morphisms out of the empty set.

\subsection{Successor case}
\label{subsec:inductive-step}

\noindent
For the successor step of the recursive definition, we fix a natural number
\(n\in\Nat\) and we assume that the structures above have been defined for \(n\)
and they satisfy the given properties. We then proceed to define them for
\(n+1\).

\paragraph{Computads~\ref{item:iComp} and their truncations~\ref{item:tr}}
An \emph{\((n\!+\!1)\)\=/computad with invertible generators} \(C\) is a tuple
consisting of the following structures:
\begin{itemize}
    \item an \(n\)\=/computad with invertible generators \(C_n\in\iComp_n\),
    \item a pair of sets \(V_{n+1}^C\) and \(U_{n+1}^C\),
    \item a function
        \(\phi_{n+1}^C\colon V_{n+1}^C\to \Par_n(\iCell_n(C_n))\),
    \item a function
        \(\psi_{n+1}^C\colon U_{n+1}^C\to \Par_n(\iCell_n(C_n))\)
\end{itemize}
We call the elements of \(V_{n+1^C}\) and \(U_{n+1}^C\) the \emph{ordinary
generators} and \emph{invertible generators} of \(C\) respectively, and we call
the functions \(\phi_{n+1}^C\) and \(\psi_{n+1}^C\) that assign a source and
target cell to each generator, the \emph{attaching functions}. The truncation
functors \(\tr^{n+1}_n\) sends a computad \(C\) to  \(C_n\).

\paragraph{Invertibility set~\ref{item:iSet}}
The \emph{invertibility set} \(\iSet_{n+1}(C)\) for a computad with invertible
generators \(C\) is defined inductively by the following:
\begin{itemize}
    \item There is an element \(\igen_{n+1}(u)\!\in\!\iSet_{n+1}(C)\) for each
        \(u\!\in\! U_{n+1}^C\).
    \item There is an element \(\lunit_{n+1}(i)\!\in\!\iSet_{n+1}(C)\)
        for each \(i\!\in\!\iSet_n(C_n)\).
    \item There is an element \(\runit_{n+1}(i)\!\in\!\iSet_{n+1}(C)\)
        for each \(i\!\in\!\iSet_n(C_n)\).
\end{itemize}
Equivalently, the set \(\iSet_{n+1}(C)\) is the coproduct of \(U_{n+1}^C\) with
two copies of \(\iSet_n(C_n)\) and we have named the three coproduct inclusions
\(\igen_{n+1}\), \(\lunit_{n+1}\), and \(\runit_{n+1}\) respectively. We define
also an auxiliary function, assigning to each element of the invertibility set,
a source and target cell by letting:
\begin{equation}\label{eq:ibdry-def}
    \begin{aligned}
        \ibdry_{n,C}
            \colon \iSet_{n+1}(C)&\to \Par_n(\iCell_n(C_n)) \\
        \ibdry_{n,C}(\igen_{n+1}(u))
            &= \psi_{n+1}^C(u) \\
        \ibdry_{n,C}(\lunit_{n+1}(i))
            &= (\linv_n(i)*^n_{n-1}\fwd_n(i),\idCell_n(\tgt_{n-1}(\fwd_n(i))))\\
        \ibdry_{n,C}(\runit_{n+1}(i))
            &= (\fwd_n(i)*^n_{n-1}\rinv_n(i),\idCell_n(\src_{n-1}(\fwd_n(i))))
    \end{aligned}
\end{equation}
Since \(\iSet_0\) is constant on the empty set, the composition operation in
the last clauses are only used when \(n>0\), in which case
\(\compCell_{n-1}^n\) is defined. Moreover, the cells being composed are indeed
composable by Property~\ref{item:prop-src-tgt-fwd-linv-rinv}, and their
composites are parallel to the identity cells by
Property~\ref{item:prop-src-tgt-comp-idcell}.

\paragraph{Inclusion of globular sets~\ref{item:iCptd}}
Every \({(n+1)}\)\=/globular set \(X\) gives rise to an
\({(n+1)}\)\=/computad \(\iCptd_{n+1}(X)\) consisting of:
\begin{itemize}
    \item underlying \(n\)\=/computad \(\iCptd_n(\tr_n(X))\),
    \item set of ordinary generators \(V_{n+1}^{\iCptd X} = X_{n+1}\),
    \item no invertible generators \(U_{n+1}^{\iCptd X} = \emptyset\),
    \item attaching function \(\phi_{n+1}^{\iCptd X}\) given by the composite
        \begin{equation}\label{eq:def-attaching-icptd}
            \begin{tikzcd}[column sep = small ]
                & {\Par_n(\tr_n X)} \\
                {X_{n+1}} && {\Par_n(\iCell_n(\iCptd_n(\tr_n(X))))}
                \arrow[from=2-1,to=1-2, "{(\src_n,\tgt_n)}"]
                \arrow[from=1-2,to=2-3, "{\Par_n(\unitCell_{n,\tr_n X})}"]
                \arrow[from=2-1,to=2-3,dashed, "{\phi_{n+1}^{\iCptd X}}"]
            \end{tikzcd}
        \end{equation}
    \item attaching function \(\psi_{n+1}^{\iCptd X}\) the unique function out
        of the empty set.
\end{itemize}
It follows immediately from the definition and the inductive hypothesis that the
right square of~\eqref{eq:truncation-cell} commutes on objects for \(m = n\).

\paragraph{Top-dimensional cells~\ref{item:iCell} and
morphisms~\ref{item:iComp}}
The following part of the definition is where we use the type-theoretic
principle of \emph{induction-recursion}~\cite{
dybjerInductionRecursionInitial2003,hancockSmallInductionRecursion2013}.
This allows us to define simultaneously the set of top-dimensional cells of an
\((n+1)\)\=/computad \(C\) with an auxiliary \emph{boundary function}, and the
sets of morphisms with target \(C\) with their \(n\)\=/truncations as the
carrier of the initial algebra for a polynomial endofunctor on a category of
families of sets indexed by a class. Existence of the initial algebra will be
proven in Theorem~\ref{thm:initial-algebra-exists} below.

The set of top-dimensional \emph{cells} \(\iCell_{n+1}(C)_{n+1}\) of an
\((n+1)\)\=/computad \(C\) is defined inductively by the following five
constructors:
\begin{itemize}
    \item There exists a cell \(\gen_{n+1}(v)\) for every \(v\in V_{n+1}^C\).
    \item There exists a cell \(\fwd_{n+1}(i)\) for every
        \(i\in \iSet_{n+1}(C)\).
    \item There exists a cell \(\linv_{n+1}(i)\) for every
        \(i\in \iSet_{n+1}(C)\).
    \item There exists a cell \(\rinv_{n+1}(i)\) for every
        \(i\in \iSet_{n+1}(C)\).
    \item There exists a cell \(\coh_{B,A}[f]\) for every triple consisting of
        \begin{itemize}
            \item a Batanin tree \(B\) of dimension at most \(n+1\),
            \item a full sphere \(A\in \Full_n(B)\), and
            \item a morphism \(f\colon \iCptd_{n+1}(\Pos(B))\to C\).
        \end{itemize}
\end{itemize}
The first constructor embeds ordinary generators into cells, while the second
one embeds the invertible generators and the unitors of the inverses into cells.
The third and fourth constructors adjoint a left and right inverse to each one
of the invertible cells. Finally, the last constructor is responsible for the
composition operations of \(\omega\)\=/categories (when \(\dim B = n+1\)) and
their coherences (when \(\dim B \le n\)).

The source and target of these cells will be determined by the \emph{boundary
function} defined recursively by the following equations:
\begin{equation}
    \begin{aligned}
        \bdry_{n,C}\colon \iCell_{n+1}(C)_{n+1} &\to \Par_n(\iCell_n(C_n)) \\
        \bdry_{n,C}(\gen_{n+1}(v)) &= \phi_{n+1}^C(v) \\
        \bdry_{n,C}(\fwd_{n+1}(i)) &= \ibdry_{n,C}(i) \\
        \bdry_{n,C}(\linv_{n+1}(i)) &= \ibdry_{n,C}(i)^{\op} \\
        \bdry_{n,C}(\rinv_{n+1}(i)) &= \ibdry_{n,C}(i)^{\op} \\
        \bdry_{n,C}(\coh_{B,A}[f]) &= \Par_n(\iCell_n(\tr^{n+1}_n f))(A)
    \end{aligned}
\end{equation}
where \(\op\colon \Par_n(X)\to \Par_n(X)\) is the natural involution
swapping the order of the two elements of a pair of parallel cells.

The set of \emph{morphisms of computads} \(\iComp_{n+1}(D,C)\) is defined
inductively by letting a morphism \(f\colon D\to C\) be a triple consisting of:
\begin{itemize}
    \item a morphism of \(n\)\=/computads \(f_n\colon D_n\to C_n\)
    \item a function \(f_V\colon V_{n+1}^D\to \iCell_{n+1}(C)\)
    \item a function \(f_U\colon U_{n+1}^D\to \iSet_{n+1}(C)\)
\end{itemize}
subject to the commutativity of the following squares:
\begin{equation}\label{eq:def-morphism}
    \begin{aligned}
        \begin{tikzcd}[ampersand replacement=\&, column sep = 5.5em]
            {V_{n+1}^D} \& {\iCell_{n+1}(C)} \\
            {\Par_n(\iCell_n(D_n))} \& {\Par_n(\iCell_n(C_n))}
            \arrow["{f_V}", from=1-1, to=1-2]
            \arrow["{\phi_{n+1}^D}"', from=1-1, to=2-1]
            \arrow["{\bdry_{n,C}}", from=1-2, to=2-2]
            \arrow["{\Par_n(\iCell_n(f_n))}"', from=2-1, to=2-2]
        \end{tikzcd} \\
        \begin{tikzcd}[ampersand replacement=\&,column sep = 5.5em]
            {U_{n+1}^D} \& {\iSet_{n+1}(C)} \\
            {\Par_n(\iCell_n(D_n))} \& {\Par_n(\iCell_n(C_n))}
            \arrow["{f_U}", from=1-1, to=1-2]
            \arrow["{\psi_{n+1}^D}"', from=1-1, to=2-1]
            \arrow["{\ibdry_{n,C}}", from=1-2, to=2-2]
            \arrow["{\Par_n(\iCell_n(f_n))}"', from=2-1, to=2-2]
        \end{tikzcd}
    \end{aligned}
\end{equation}
The \(n\)\=/truncation of a morphism \(f\) is given by the first projection
\(\tr_n^{n+1}(f) = f_n\). This concludes the inductive-recursive portion of the
definition.

\paragraph{Functorial action of the invertibility set~\ref{item:iSet}}
A morphism of \((n+1)\)\=/computads \(f\colon D\to C\) gives rise to a function
on invertibility sets recursively by:
\begin{equation}
    \begin{aligned}
        \iSet_{n+1}(f) \colon \iSet_{n+1}(D) &\to \iSet_{n+1}(C) \\
        \iSet_{n+1}(f)(\igen_{n+1}(u)) &= f_U(u) \\
        \iSet_{n+1}(f)(\lunit_{n+1}(i)) &= \lunit_{n+1}(\iSet_n(f_n)(i)) \\
        \iSet_{n+1}(f)(\runit_{n+1}(i)) &= \runit_{n+1}(\iSet_n(f_n)(i))
    \end{aligned}
\end{equation}
The \(\ibdry_n\) functions are natural, that is, the following square commutes:
\begin{equation}\label{eq:ibdry-natural}
    \begin{tikzcd}[ampersand replacement=\&,column sep = 5.5em]
        {\iSet_{n+1}(D)} \& {\iSet_{n+1}(C)} \\
        {\Par_{n}(\iCell_n(C))} \& {\Par_{n}(\iCell_n(C))}
        \arrow["{\iSet_{n+1}(f)}", from=1-1, to=1-2]
        \arrow["{\ibdry_{n,D}}"', from=1-1, to=2-1]
        \arrow["{\ibdry_{n,C}}", from=1-2, to=2-2]
        \arrow["{\Par_n(\iCell_n(f_n))}"', from=2-1, to=2-2]
    \end{tikzcd}
\end{equation}
For invertible generators, this follows by commutativity of
the second square in~\eqref{eq:def-morphism}, while in the other two cases, it
follows by naturality of the composition operations and identities, as well as
that of of \(\fwd_n\), \(\linv_n\) and \(\rinv_n\).

\paragraph{Composition~\ref{item:iComp} and functorial action of
    cells~\ref{item:iCell}}
Let \(f\colon D\to C\) be a morphism of computads. The function
\(\iCell_{n+1}(f)_{n+1}\) on top-dimensional cells is defined mutually
recursively with the post-composition operation \(f\circ -\) together with a
proof that \(\bdry_{n}\) is natural and \(\tr_n^{n+1}\) is a functor. The former
is given by:
\begin{equation}\label{eq:iCell-morphisms}
    \begin{aligned}
        \iCell_{n+1}(f)_{n+1}\colon
            \iCell_{n+1}(D)_{n+1}&\to \iCell_{n+1}(C)_{n+1} \\
        \iCell_{n+1}(f)_{n+1}(\gen_{n+1}(v))
            &= f_V(v) \\
        \iCell_{n+1}(f)_{n+1}(\fwd_{n+1}(i))
            &= \fwd_{n+1}(\iSet_{n+1}(f)(i)) \\
        \iCell_{n+1}(f)_{n+1}(\linv_{n+1}(i))
            &= \linv_{n+1}(\iSet_{n+1}(f)(i)) \\
        \iCell_{n+1}(f)_{n+1}(\rinv_{n+1}(i))
            &= \rinv_{n+1}(\iSet_{n+1}(f)(i)) \\
        \iCell_{n+1}(f)_{n+1}(\coh_{B,A}[g])
            &= \coh_{B,A}[f\circ g]
    \end{aligned}
\end{equation}
where in the last clause we may assume that \(f\circ g\) has been defined by
structural recursion. To show that the following square commutes:
\begin{equation}\label{eq:bdry-natural}
    \begin{tikzcd}[ampersand replacement=\&,column sep = 5.5em]
        {\iCell_{n+1}(D)_{n+1}} \& {\iCell_{n+1}(C)_{n+1}} \\
        {\Par_{n}(\iCell_n(C))} \& {\Par_{n}(\iCell_n(C))}
        \arrow["{\iCell_{n+1}(f)_{n+1}}", from=1-1, to=1-2]
        \arrow["{\bdry_{n,D}}"', from=1-1, to=2-1]
        \arrow["{\bdry_{n,C}}", from=1-2, to=2-2]
        \arrow["{\Par_n(\iCell_n(f_n))}"', from=2-1, to=2-2]
    \end{tikzcd}
\end{equation}
we proceed recursively as well. For ordinary generators, this amounts to the
commutativity of the first square in~\eqref{eq:def-morphism}. For the three
cases involving the invertibility set, it follows by naturality of \(\ibdry\)
and of \(\op\). Finally, for the case \(\coh_{B,A}[f]\), we may assume further
that \(\tr_n\) preserves the composition \(f\circ g\) and use functoriality of
\(\Par_n\circ\iCell_n\).

To define the composition of \(f\) with a morphism \(g\colon E\to D\), we may
assume by structural induction that \(\iCell_{n+1}(f)_{n+1}\) has been defined
on cells of the form \(g_V(v)\) and that the naturality
square~\eqref{eq:bdry-natural} commutes on these cells. We then define
\begin{equation}\label{eq:def-composition}
    \begin{aligned}
        (f\circ g)_n &= f_n\circ g_n \\
        (f\circ g)_V(v) &= \iCell_{n+1}(f)_{n+1}(g_V(v)) \\
        (f\circ g)_U(u) &= \iSet_{n+1}(f)(g_U(u))
    \end{aligned}
\end{equation}
This is a well-defined morphism, in that the it makes~\eqref{eq:def-morphism}
commute, by the naturality assumption on \(\bdry_n\) and \(\ibdry_n\). The
first of the defining equations shows that \(tr^{n+1}_n\) preserves the
composition.

\paragraph{Associativity~\ref{item:iComp} and functoriality of
    ivertibility set~\ref{item:iSet} and cells~\ref{item:iCell}}
Suppose that a pair \(g\colon E\to D\) and \(f\colon D\to C\) of morphisms of
\((n+1)\)\=/computads are given. Then we can verify that
\begin{equation}\label{eq:iSet-functorial}
    \iSet_{n+1}(f\circ g) = \iSet_{n+1}(f) \circ \iSet_n(g)
\end{equation}
by cases. This equality of functions holds on invertible generators by
definition of composition~\eqref{eq:def-composition}, while it follows on left
and right units by the fact that \(\tr^{n+1}_n\) preserves composition.

The verification of the analogous equation for cells:
\begin{equation}\label{eq:iCell-functorial}
    \iCell_{n+1}(f\circ g)_{n+1} = \iCell_{n+1}(f)_{n+1} \circ \iCell_n(g)_{n+1}
\end{equation}
requires mutual induction with associativity of composition:
\begin{equation}\label{eq:iComp-assoc}
    (f\circ g) \circ h = f\circ (g\circ h)
\end{equation}
Equation~\eqref{eq:iCell-functorial} holds on generator by the definition of
composition~\eqref{eq:def-composition}, while it holds on cells built using the
invertibility set by~\eqref{eq:iSet-functorial}. Finally, to show that it holds
for cells of the form \(\coh_{B,A}[h]\), we may assume by structural induction
that~\eqref{eq:iComp-assoc} holds for \(h\colon \iCptd_{n+1}(\Pos(B))\to E\).
Finally, to show~\eqref{eq:iComp-assoc} holds for arbitrary morphism
\(h\colon F\to E\), we may assume that~\eqref{eq:iCell-functorial} holds for
cells of the form \(h_V(v)\) for \(v\in V_{n+1}^F\), which
reduces~\eqref{eq:iComp-assoc} to a simple computation.

\paragraph{Identities~\ref{item:iComp} and functoriality of
    ivertibility set~\ref{item:iSet} and cells~\ref{item:iCell}}
Suppose now that an \(n+1\)\=/computad \(C\) is fixed. We define the identity on
\(C\) to be the triple \(\id_C = (\id_{C_n}, \gen_{n+1}, \igen_{n+1})\). It
follows from the definition that \(\iSet_{n+1}\) preserves the identities
and that \(f\circ \id_C = f\) for every morphism with target \(C\). To show that
the other unit law for identities and that \(\iCell_{n+1}(-)_{n+1}\) preserves
identities, one proceeds mutually inductively as in the previous paragraph. From
the work up to this point, we can conclude that \(\iComp_{n+1}\) is a
well-defined category and that \(\iCell_{n+1}(-)_{n+1}\), \(\tr^{n+1}_n\) and
\(\iSet_{n+1}\) are well-defined functors out of it.

\paragraph{Truncation functors~\ref{item:tr}}
We have already defined the functor \(\tr^{n+1}_n\). The remaining truncation
functors for \(m < n\) are defined by
\begin{equation}\label{eq:low-truncation}
    \tr_m^{n+1} = \tr^n_m \circ\tr^{n+1}_n
\end{equation}
These functors satisfy the compatibility condition of
Property~\ref{item:prop-cocycle} by the inductive assumption that this property
holds for \(n\).

\paragraph{Globular set of cells~\ref{item:iCell}}
The \((n+1)\)\=/globular set of cells \(\iCell_{n+1}(C)\) of an
\((n+1)\)\=/computad \(C\) consists of the inductively defined set above
together with the sets
\begin{equation}\label{eq:low-dim-cells}
    \iCell_{n+1}(C)_{m} = \iCell_n(C_n)_m
\end{equation}
for \(m\le n\). The source and target of an \(m\)\=/cell in \(\iCell_{n+1}(C)\)
are given by those in \(\iCell_{n}(C_n)\) for \(m\le n\), and they are given by
the boundary function \(\bdry_{n,C}\) for \(m = n+1\).

A morphism \(f\colon D\to C\) gives rise to a morphism of globular sets via the
recursively defined function of~\eqref{eq:iCell-morphisms} and the functions
comprising \(\iCell_n(f_n)\). This is a well-defined morphism of globular sets
by the naturality of \(\bdry_{n,C}\), and it is functorial.

This assignment clearly commutes with the truncation functor \(\tr^{n+1}_n\),
meaning that the left square of~\ref{eq:truncation-cell} commutes when \(m=n\).
To show that it also commutes for \(m<n\), we may use the inductive hypothesis
together with Property~\ref{item:prop-cocycle}. We note furthermore that by the
definition of \(\iCell_{n+1}(-)_{n+1}\) on morphisms, the constructors
\(\fwd_{n+1}\), \(\linv_{n+1}\) and \(\rinv_{n+1}\) become natural
transformations, as expected by~\ref{item:fwd-linv-rinv}. Moreover, these
transformations satisfy Property~\ref{item:prop-src-tgt-fwd-linv-rinv} by the
definition of the boundary natural transformations.

\paragraph{Functoriality of the inclusion of globular sets~\ref{item:iCptd}}
A morphism of \(({n+1})\)\=/globular sets \(f\colon X\to Y\) gives rise to a
morphism of computads \(\iCptd_{n+1}(f)\) by letting
\begin{equation}\label{eq:def-icptd-morphism}
    \begin{aligned}
        \iCptd_{n+1}(f)_n &= \iCptd_n(\tr^{n+1}_n f) \\
        \iCptd_{n+1}(f)_V &= \gen_{n+1}\circ\ f_{n+1} \\
        \iCptd_{n+1}(f)_U &=\ !
    \end{aligned}
\end{equation}
where \(!\) is the unique function out of the empty set. This is a well-defined
morphism, in that the squares of~\eqref{eq:def-morphism} commute, due to
naturality of the transformation \(\unitCell_n\). Unwrapping the definition of
composition, we get that \(\iCptd_{n+1}\) is a functor. Moreover,
by definition, \(\iCptd_{n+1}\) commutes with the truncation functor
\(\tr^{n+1}_{n}\). By the inductive hypothesis and
Property~\ref{item:prop-cocycle}, we may therefore conclude that the right
square of~\eqref{eq:truncation-cell} commutes for all \(m\le n\).

\paragraph{The unit~\ref{item:unitCell} and counit~\ref{item:counitCell}}
The unit of the adjunction \(\unitCell_{n+1}\) is the natural transformation
obtained by combining the natural transformation
\(\unitCell_n\circ\tr_n^{n+1}\) with the inclusion \(\gen_{n+1}\iCptd_{n+1}\)
of ordinary generators into cells. More precisely, its component on an
\((n+1)\)\=/globular set \(X\) is the morphism of \((n+1)\)\=/globular sets
\({\unitCell_{n+1,X}\colon X\to \iCell_{n+1}\iCptd_{n+1}(X)}\)
sending a cell \(x\in X_m\) to:
\begin{equation}\label{eq:def-unitCell}
    \unitCell_{n+1,X}(x) = \begin{cases}
        \gen_{n+1,\iCptd_{n+1}X}(x),& \text{if } m = n+1 \\
        \unitCell_{n,\tr_n X}(x),& \text{if } m \le n
    \end{cases}
\end{equation}
This is a well-defined morphism of globular sets by the definition of the
attaching function \(\phi_{n+1}^{\iCptd X}\) in~\eqref{eq:def-attaching-icptd}
and the boundary natural transformation. Its naturality follows from that
of \(\unitCell_n\), and the definition of \(\iCptd_{n+1}\) on
morphisms~\eqref{eq:def-icptd-morphism}.

The counit of the adjunction \(\counitCell_{n+1}\) is the natural transformation
obtained by combining \(\counitCell_n\circ\tr_n^{n+1}\) together with the
identity of the set of cells. More precisely, it sends an \((n+1)\)\=/computad
\(C\) to the morphism of \((n+1)\)\=/computads
\(\counitCell_{n+1,C}\colon \iCptd_{n+1}\iCell_{n+1}(C)\to C\) given by:
\begin{equation}\label{eq:def-counitCell}
    \begin{aligned}
        (\counitCell_{n+1,C})_n &= \counitCell_{n,C_n} \\
        (\counitCell_{n+1,C})_V &= \id_{\iCell_{n+1}(C)_{n+1}} \\
        (\counitCell_{n+1,C})_V &=\ !
    \end{aligned}
\end{equation}
where again \(!\) is the unique function out of the empty set. To show that this
is a well-defined morphism of computads, we need to see that the exterior of the
following diagram commutes:
\begin{equation}\label{eq:def-counitCell-welldef}
    \begin{tikzcd}[column sep = small, cramped]
        &[-4em] {V_{n+1}^{\iCptd\iCell C}} & {\iCell_{n+1}(C)} \\
        &[-4em] {\Par_n\iCell_n(C_n)} & {\Par_n(\iCell_n(C_n))}\\
        \\
        {\Par_n(\iCell_n(\iCptd_n(\iCell_n(C_n))))}
        \arrow[equals, from=1-2, to=1-3]
        \arrow["{\bdry_{n,C}}"', from=1-2, to=2-2]
        \arrow["{\phi_{n+1}^{\iCptd\iCell C}}"',
            bend right = 40, from=1-2, to=4-1]
        \arrow["{\bdry_{n,C}}", from=1-3, to=2-3]
        \arrow["{\Par_n(\unitCell_{n,\iCell_nC_n})}"{description},
            from=2-2, to=4-1]
        \arrow[equals, from=2-2, to=2-3]
        \arrow["{\Par_n(\iCell_n(\counitCell_{n,C_n}))}"{swap}, bend right = 10,
            from=4-1, to=2-3]
    \end{tikzcd}
\end{equation}
The left triangle commutes by definition of the attaching map, while the bottom
one commutes by the triangle equations~\eqref{eq:trianglesCell}, from which we
get commutativity of the whole diagram.

By construction, both the unit and the counit commute with the truncation
functors \(\tr^{n+1}_n\), so by the inductive hypothesis and
Property~\ref{item:prop-cocycle}, we see that the
equations~\eqref{eq:truncation-unit} hold for all \(m\le n\). This concludes the
proof of Property~\ref{item:prop-tr-adjunction}. To show
Property~\ref{item:prop-adjunction}, it remains to check that the triangle
identities~\eqref{eq:trianglesCell} are satisfied and that \(\counitCell_{n+1}\)
is natural. By compatibility with the truncation functors and the fact that the
trinagle equations hold one dimension below, the first triangle
identity amounts to showing for every \((n+1)\)\=/computad \(C\) and every
cell \(c\in \iCell_{n+1}(C)_{n+1}\) that
\begin{equation}\label{eq:trianglesCell-first}
    \iCell_{n+1}(\counitCell_{n+1,C}) \circ \unitCell_{n+1,\iCell_{n+1}C}(c) =
    (\counitCell_{n+1,C})_V(c) = c
\end{equation}
which follows from the definition of the unit and the counit. Similarly, the
second triangle identity amounts to showing for every \((n+1)\)\=/truncated
globular set \(X\) and every \(x\in X_{n+1}\) that:
\begin{equation}\label{eq:trianglesCell-second}
    (\counitCell_{n+1,\iCptd_{n+1}X}\circ\iCptd_{n+1}(\unitCell_{n+1,X}))_V(x)
        = \gen_{n+1}(x)
\end{equation}
which follows from equation~\eqref{eq:trianglesCell-first} for
\(C = \iCptd_{n+1}X\) and \(c = \gen_{n+1}(x)\). Finally, for the same reason,
to check naturality of the counit with respect to a morphism \(f\colon D\to C\),
it suffices to show for every \(d\in \iCell_{n+1}(D)_{n+1}\) that
\begin{equation}\label{eq:counitCell-naturality}
    \iCell_{n+1}(\counitCell_{n+1,C})(\gen_{n+1}(\iCell_{n+1}(f)(d)))
        = \iCell_{n+1}(f)(d)
\end{equation}
which is again an instance of~\eqref{eq:trianglesCell-first} by letting
\(c = \iCell_{n+1}(f)(d)\).

\paragraph{Fullness~\ref{item:Full}}
To define when a pair of parallel cells is \emph{full}, we first introduce the
notion of \emph{support} of invertible cells and of cells respectively. The
support of an invertible cell of \(C\) is a subset of the invertible generators
of \(C\) defined recursively by:
\begin{equation}\label{eq:def-isupp}
    \begin{aligned}
        \isupp_{n+1}\colon \iSet_{n+1}(C)&\to \PSet(U_{n+1}^C) \\
        \isupp_{n+1}(\igen_{n+1}(u)) &= \set{u} \\
        \isupp_{n+1}(\lunit_{n+1}(i)) &= \emptyset \\
        \isupp_{n+1}(\runit_{n+1}(i)) &= \emptyset
    \end{aligned}
\end{equation}
The support of an arbitrary cell is the subset of all generators of \(C\)
defined recursively by:
\begin{equation}\label{eq:def-supp}
    \begin{aligned}
        \supp_{n+1}\colon \iSet_{n+1}(C)&\to \PSet(U_{n+1}^C \amalg V_{n+1}^C) \\
        \supp_{n+1}(\gen_{n+1}(v)) &= \set{v} \\
        \supp_{n+1}(\fwd_{n+1}(i)) &= \isupp_{n+1}(i) \\
        \supp_{n+1}(\linv_{n+1}(i)) &= \isupp_{n+1}(i) \\
        \supp_{n+1}(\rinv_{n+1}(i)) &= \isupp_{n+1}(i) \\
        \supp_{n+1}(\coh_{B,A}[f])
            &= \bigcup_{p\in \Pos_{n+1}(B)} \supp_{n+1}(f_V(v))
    \end{aligned}
\end{equation}
We may then define an \((n+1)\)\=/sphere \((a,b)\) of a Batanin tree \(B\) to
be \emph{full} when:
\begin{itemize}
    \item \(\supp_{n+1}(a)
        = \set{\sigma_{n+1}^B(p) \such p\in \Pos_{n+1}(\partial_{n+1}B)}\),
    \item \(\supp_{n+1}(b)
        = \set{\tau_{n+1}^B(p) \such p\in \Pos_{n+1}(\partial_{n+1}B)}\), and
    \item the \(n\)\=/sphere \(\bdry_{n}(a) = \bdry_n(b)\) is full.
\end{itemize}
As pointed out by Dean et al.~\cite{deanComputadsWeak$omega$categories2024},
this fullness condition derives from the admissibility condition of
Maltsiniotis~\cite{maltsiniotisGrothendieck$infty$groupoidsStill2010}. The cell
\(\coh_{B,A}[\id]\) is the liftings of the admissible pair corresponding to
the sphere \(A\).

\paragraph{Identity operations~\ref{item:idCell}}
Using the adjunction \(\iCptd_{n+1}\dashv \iCell_{n+1}\) and the Yoneda
lemma, we get a natural isomorphism
\begin{equation}\label{eq:fin-disks-represent-cells}
    \chi_{\bullet}\colon\iCell_{n+1}(C)_k \xlongrightarrow{\sim}
        \iComp_{n+1}(\iCptd_{n+1}(\Pos(D_k)), C)
\end{equation}
for all \((n+1)\)\=/computads \(C\) and \(k\le n+1\). Using this isomorphism and
the full spheres of Property~\ref{item:prop-full}, we may define for every
\(c\in\iCell_{n+1}(C)_n\),
\begin{equation}\label{eq:idCell-def}
    \idCell_{n+1}(c) = \coh_{D_n, A^{\id}_{n,n}}[\chi_c]
\end{equation}
Naturality of the isomorphism \(\chi_\bullet\) shows that the identity operation
is natural, in that for every \(f\colon C\to D\),
\begin{equation}\label{eq:idCell-nat}
    \iCell_{n+1}(f)(\idCell_{n+1}(c)) = \idCell_{n+1}(\iCell_{n+1}(f)(c))
\end{equation}
Property~\ref{item:prop-src-tgt-comp-idcell} about the boundary of identity
cells follows by evaluation at the cell \(\gen_{n+1}(\id_{k})\) being the
inverse of the isomorphism \(\chi_\bullet\).
For Property~\ref{item:prop-full}, we note that the sphere \(A_{n+1,k}^{\id}\)
of \(D_k\) for \(n+1\le k\) are well-defined with both their cells having
boundary \(A_{n,k}^{\id}\). Fullness of these spheres can be checked immediately
from the definition of fullness and naturality of the isomorphisms
\(\Pos(D_n)\cong\Disk^n\) shown in~\eqref{eq:linear-trees-src-tgt}.

\paragraph{Composition operations~\ref{item:bin-comp}}
Using again the same adjunction and the isomoprhisms
in~\ref{eq:binary-trees-represent}, we get for all \(k < n+1\) a natural
isomorphism \(\chi^k_{\bullet,\bullet}\) between pairs of \(k\)\=/composable
\((n+1)\)\=/cells and morphisms out of \(\iCptd_{n+1}(P_{n+1}^k)\). We may
thus define the \(k\)\=/composable operation by
\begin{equation}\label{eq:compCell-def}
    a\compCell^k_{n+1} b = \coh_{P^k_{n+1}, A^{\comp}_{n,n,k}}[\chi^k_{a,b}]
\end{equation}
This is well-defined and has the source and target claimed by
Property~\ref{item:prop-src-tgt-comp-idcell}, since the inverse of
\(\chi^k_{\bullet,\bullet}\) is given by evaluation at the two top-dimensional
cells of \(P_n^k\). To finish the induction, it remains to show that the sphere
\(A_{n+1,m,k}^{\comp}\) of Property~\ref{item:prop-full} is well-defined and
full for all \(n+1\le m<\omega\) and \(k<m\). To show that it is well-defined,
we use Property~\ref{item:prop-src-tgt-comp-idcell} one dimension lower to
see that both cells comprising \(A_{n+1,m,k}^{\comp}\) have boundary
\(A_{n,m,k}^{\comp}\). To see that it is full, we further use that the support
of the \(a \compCell^{n+1}_k b\) in the notation of
Property~\ref{item:prop-full} contains exactly the two top-dimensional cells
of \(P_{n+1}^k = \partial_{n+1} P_m^k\).

\begin{remark}\label{rmk:poly-endo-ind-rec}
    Before moving to the limit case, we expand upon the
    inductive-recursive component of the definition, defining the cells of
    a computad \(C\) and morphisms with target \(C\). Since these are defined
    together with a boundary function and a truncation function, we understand
    the definition as specifying an object of the large category:
    \begin{equation}\label{eq:poly-endo-category}
        \mathcal{C} = \left(\Set \downarrow \Par_n(\iCell_n(C))\right)
        \times \prod_{D\in \iComp_{n+1}} (\Set \downarrow \iComp_n(D_n,C_n))
    \end{equation}
    whose objects \((X,f,(X_D,f_D))\) consist of a pair of a set \(X\) with
    a function \(f\colon X\to \Par_n(\iCell_n(C))\), and a family of sets
    \(X_D\) equipped with functions \(f_D\colon X_D\to \iComp_{n}(D_n,C_n)\)
    indexed by computads \(D\in \iComp_{n+1}\).

    The constructors for cells and morphisms
    specify extra structure that this object must possess, namely that is is an
    algebra for the endofunctor \(F\colon \mathcal{C}\to \mathcal{C}\) sending
    an object \((X,f,(X_D,f_D))\) as above to the object \((X',f',(X'_D,f'_D))\)
    defined below. The set \(X'\) is the coproduct:
    \begin{equation}\label{eq:poly-endo-set-cells}
            X' = V_{n+1}^C \amalg \left(\coprod_{i=1}^3 \iSet_{n+1}(C)\right)
            \amalg \left(\coprod_{\substack{\dim B \le n+1 \\ A\in \Full_n(B)}}
            X_{\iCptd_{n+1}(\Pos(B))}\right)
    \end{equation}
    On the first component, \(f'\) is defined to agree with \(\phi_{n+1}^C\).
    On the second comoponent, it is given by the functions \(\ibdry_{n,C}\),
    \(\op\ibdry_{n,C}\) and \(\op\ibdry_{n,C}\), which have already been
    defined. On the last component, it sends
    \(g\in X_{\iCptd_{n+1}(\Pos(B))}\) from the component \((B,A)\) to
    \(\Par_n(\iCell_{n+1}(f_{\iCptd_{n+1}(\Pos(B))}(g)))(A)\). The set \(X'_D\)
    consists of triples \((g_n,g_V,g_U)\) where:
    \begin{align}\label{eq:poly-endo-set-morphisms}
        g_n&\colon D_n\to C_n &
        g_V&\colon V_{n+1}^D\to X &
        g_U&\colon U_{n+1}^D\to \iSet_{n+1}(C)
    \end{align}
    subject to the commutativity of the following squares:
    \begin{equation}\label{eq:poly-endo-set-morphisms-contraints}
        \begin{aligned}
            \begin{tikzcd}[ampersand replacement=\&, column sep = 5.5em]
                {V_{n+1}^D} \& {X} \\
                {\Par_n(\iCell_n(D_n))} \& {\Par_n(\iCell_n(C_n))}
                \arrow["{g_V}", from=1-1, to=1-2]
                \arrow["{\phi_{n+1}^D}"', from=1-1, to=2-1]
                \arrow["{f}", from=1-2, to=2-2]
                \arrow["{\Par_n(\iCell_n(g_n))}"', from=2-1, to=2-2]
            \end{tikzcd} \\ \\
            \begin{tikzcd}[ampersand replacement=\&,column sep = 5.5em]
                {U_{n+1}^D} \& {\iSet_{n+1}(C)} \\
                {\Par_n(\iCell_n(D_n))} \& {\Par_n(\iCell_n(C_n))}
                \arrow["{g_U}", from=1-1, to=1-2]
                \arrow["{\psi_{n+1}^D}"', from=1-1, to=2-1]
                \arrow["{\ibdry_{n,C}}", from=1-2, to=2-2]
                \arrow["{\Par_n(\iCell_n(g_n))}"', from=2-1, to=2-2]
            \end{tikzcd}
        \end{aligned}
    \end{equation}
    This endofunctor is quite similar to the one appearing in the definition of
    computads by Dean et al.~\cite[Proposition~3.1]{
    deanComputadsWeak$omega$categories2024}. The only difference is the addition
    of the three copies of the invertibility set that jointly provide inverses
    for the invertible generators.

    Being an \emph{inductive} definition, we understand that the defined object
    is precisely the carrier of the initial algebra of the endofunctor \(F\),
    the existence of which we prove in the
    Theorem~\ref{thm:initial-algebra-exists}.
\end{remark}

\begin{theorem}\label{thm:initial-algebra-exists}
    The endofunctor \(F\colon \mathcal{C}\to \mathcal{C}\) of
    Remark~\ref{rmk:poly-endo-ind-rec} has an initial algebra.
\end{theorem}

\begin{proof}
    Following Adamek's construction~\cite{adamekFreeAlgebrasAutomata1974}, we
    define a transfinite sequence in \(c^\alpha =
    (X^\alpha,f^\alpha,(X_D^\alpha,f^\alpha_D))\in\mathcal{C}\) indexed by
    ordinals \(d\le \omega + 2\) by letting:
    \begin{align}\label{eq:initial-algebra-chain}
        c^0 &= \emptyset &
        c^{d+1} &= F(c^d) &
        c^\omega &= \colim_{d<\omega} c^d
    \end{align}
    We observe first that \((X^\omega,f^\omega)\cong
    (X^{\omega+1},f^{\omega+1})\) by commutation of coproducts with colimits:
    letting \(Y = V_{n+1}^C \amalg \coprod_{i=1}^3\iSet_{n+1}(C)\), we have
    that
    \begin{equation}
        \begin{split}
            X^{\omega+1}
                &= Y \amalg \left(
                    \coprod_{B,A}
                    X^\omega_{\iCptd_{n+1}(\Pos(B))}\right) \\
                &= Y \amalg \left(
                    \coprod_{B,A}
                    \colim_{d<\omega} X^d_{\iCptd_{n+1}(\Pos(B))}\right) \\
                &= \colim_{d<\omega} Y \amalg \left(
                    \coprod_{B,A}
                    X^d_{\iCptd_{n+1}(\Pos(B))}\right) \\
                &= \colim_{d<\omega} X^{(d+1)} = X^\omega
        \end{split}
    \end{equation}
    Unwrapping the definition of \(F\), we see that
    \((X_D^{\omega+1},f_D^{\omega+1}) \cong (X_D^{\omega+2},f_D^{\omega+2})\)
    for every \((n+1)\)\=/computad \(D\).

    We then observe that for every \((n+1)\)\=/computad \(D\) that has finitely
    many ordinary generators,
    \((X^{\omega}_D,f^{\omega}_D)\cong (X^{\omega+1}_D,f^{\omega+1}_D)\) by
    compactness of finite sets: the set of functions \(V_{n+1}^D\to X^\omega\)
    is precisely the colimit of the sets \(V_{n+1}^D\to X^d\) over all
    \(d<\omega\). Commutativity of \eqref{eq:poly-endo-set-morphisms-contraints}
    is preserved under this identification, since the inclusions
    \(X^d\to X^\omega\) into the colimit are monic. The latter follows by the
    morphisms \(X^d\to X^{d+1}\) being monic for \(d<\omega\), which can be
    checked by induction on \(d\).

    Specialising this isomorphism to computads of the form
    \(\iCptd_{n+1}(\Pos(B))\), we conclude that
    \((X^{\omega+1},f^{\omega+1})\cong
    (X^{\omega+2},f^{\omega+2})\). Combining all the isomorphisms we have shown
    so far, we get that \(c^{\alpha+1}\cong c^{\alpha+2}\). Therefore, the
    transfinite sequence we built terminates and \(c^{\alpha+1}\) is the
    carrier of the initial algebra of \(F\). The algebra structure is given by
    the isomorphism \(c^{\alpha+2} \cong c^{\alpha+1}\).
\end{proof}



\subsection{Limit case}
\label{subsec:infinite}

We may now assume that all structures have been defined for \(-1\le n < \omega\)
and that they satisfy the claimed properties. We then proceed to define these
structures for \(n = \omega\) as well. By abuse of notation, we often drop
\(\omega\) from the notation, speaking e.g. of computads instead of
\(\omega\)\=/computads.

\paragraph{Computads~\ref{item:iComp} and truncations~\ref{item:tr}}
The category of \emph{computads with invertible generators}
\(\iComp = \iComp_{\omega}\) is the limit of the categories \(\iComp_n\) for
\(n < \omega\) under the truncation functors:
\begin{equation}\label{eq:def-omega-computad}
    \begin{tikzcd}[row sep = large]
        \iComp \\
        \cdots & {\iComp_n} & \cdots & {\iComp_0} & {\iComp_{-1}}
        \arrow["{\tr_n}", from=1-1, to=2-2]
        \arrow["{\tr_0}", bend left=10, from=1-1, to=2-4]
        \arrow["{\tr_{-1}}", bend left=15, from=1-1, to=2-5]
        \arrow["{\tr_{n}^{n+1}}"', from=2-1, to=2-2]
        \arrow["{\tr_{n-1}^n}"', from=2-2, to=2-3]
        \arrow["{\tr^1_0}"', from=2-3, to=2-4]
        \arrow["{\tr^0_{-1}}"', from=2-4, to=2-5]
    \end{tikzcd}
\end{equation}
The truncation functors \(\tr_n = \tr^\omega_n\) are the projections out of the
limit - they satisfy Property~\ref{item:prop-cocycle} definitionally. In other
words, a computad with invertible generators \(C\) consists of
an \(n\)\=/computad with invertible generators \(C_n\) for every \(n\in\Nat\)
such that \(\tr^{n+1}_nC_{n+1} = C_n\). Morphisms \(f\colon C\to D\) consist
similarly of morphisms \(f_n \colon C_n \to D_n\) for all \(n\in\Nat\)
satisfying the same condition.

\paragraph{Cells~\ref{item:iCell} and globular sets~\ref{item:iCptd}}
The category of globular sets \(\gSet\) is similarly the (strict) limit of the
categories \(\gSet_n\) for \(n < \omega\). By functoriality of limits and by
Property~\ref{item:prop-tr-adjunction}, we obtain unique functors
\begin{align}\label{eq:omega-iCptd-iCell}
    \begin{aligned}
        \iCptd &\colon \gSet\to \iComp \\
        \iCptd(X)_n &= \iCptd_n(\tr_n X)
    \end{aligned}&&
    \begin{aligned}
        \iCell &\colon \iComp\to \gSet \\
        \iCell(C)_n &= \iCell_n(\tr_n C)_n
    \end{aligned}
\end{align}
compatible with the truncation functors, that is, also satisfying
Property~\ref{item:prop-tr-adjunction}.

\paragraph{The unit~\ref{item:unitCell} and counit~\ref{item:counitCell}}
Since the unit and counit of the adjunction \(\iCptd_n\dashv\iCell_n\) for
\(n< \omega\) are also compatible with the truncation functors, in the sense of
Property~\ref{item:prop-tr-adjunction}, they also lift uniquely to natural
transformations compatible with the truncations:
\begin{align}\label{eq:omega-unit-counit}
    \begin{aligned}
        \unitCell\colon \id&\Rightarrow \iCell\iCptd \\
        \unitCell_X(x) &= \unitCell_{n,\tr_n X}(x)
    \end{aligned} &&
    \begin{aligned}
        \counitCell\colon &\iCptd\iCell \Rightarrow \id \\
        \counitCell_C &= (\counitCell_{n,\tr_n C})_{n\in\Nat}
    \end{aligned} &&
\end{align}
By uniqueness of the lifts and Property~\ref{item:prop-adjunction}, we get
immediately that \(\unitCell\) and \(\counitCell\) satisfy the triangle
identities, so \(\iCptd\) is left adjoint to \(\iCell\) with unit \(\unitCell\)
and counit \(\counitCell\). This concludes the definition of computads with
invertible generators.

\subsection{Examples and properties}
\label{subsec:example-computads}

Before we study in details the monad defined by the adjunction
\(\iCptd \dashv \iCell\), we first give a few examples of computads with
invertible generators, and we present some properties of computads.

\begin{example}\label{ex:walking-equivalence}
    The \emph{walking equivalence} \(\Equiv^{n+1}\) for \(n\in\Nat\) is the
    \((n+1)\)\=/computad obtained by adjoining to \(\iCptd_{n}(\Sphere^{n})\) an
    invertible generator from the source to the target of the sphere. More
    precisely, we let
    \begin{equation}\label{eq:walking-equivalence}
        \Equiv^{n+1} =
        (\iCptd_{n}(\Sphere^{n}), \emptyset, \set{c_{n+1}}, !,
        \psi^{\Equiv}_{n+1})
    \end{equation}
    where \(\psi^{\Equiv}_{n+1}\) sends \(c_n\) to the parallel pair
    \((\gen_n(\sigma_n), \gen_n(\tau_n))\) and \(!\) denotes the unique morphism
    out of the empty set. We may
    illustate the computad \(\Equiv^{n+1}\) as follows, marking by \(\sim\)
    the invertible generators:
    \begin{align}\label{eq:walking-equivalence-fig}
    \Equiv^{1}
    &\qquad
        \begin{tikzcd}[ampersand replacement=\&]
        \bullet \ar[r,"\sim"] \& \bullet
        \end{tikzcd}
    & \Equiv^{2}
    &\qquad
        \begin{tikzcd}[ampersand replacement=\&]
        \bullet
        \ar[r, bend left = 40, ""{below, name=A}]
        \ar[r, bend right = 40, ""{above, name=B}]
        \& \bullet
        \ar[from = A, to = B, Rightarrow,"\sim"{sloped}]
        \end{tikzcd}
    \end{align}
    We can see that the computads \(\Equiv^{n+1}\) and
    \(\Cptd_{n+1}(\Disk^{n+1})\) differ only in that the unique top-dimensional
    generator of the former is invertible, so we may define a morphism from the
    latter to the former:
    \begin{equation}\label{eq:unmarked-to-marked}
        \diskInc_{n+1} = (\id,\diskInc_{n+1,V} , !)\colon
        \Cptd_{n+1}(\Disk^{n+1})\to \Equiv^{n+1}
    \end{equation}
    where \(\diskInc_{n+1,V}\) sends the unique top-dimensional generator of
    the source to the cell \(\fwd_{n+1}(\igen_{n+1}(c_n))\) corresponding to the
    unique top-dimensional generator of the target. Composing this morphism with
    the inclusion of the sphere into the disk, we get a morphism
    \begin{equation}\label{eq:equiv-inc}
        \equivInc_{n+1} = \diskInc_{n+1}\circ\iCptd_{n+1}(\sphereInc_{n+1}) =
        (\id,!, !)\colon \Cptd_{n+1}(\Sphere^{n})\to \Equiv^{n+1}
    \end{equation}
    that will be important in describing the weak \(\omega\)\=/categories
    generated by computads with invertible generators. Composing further with
    the source inclusion between disks, we get a morphism
    \begin{equation}\label{eq:trivial-cof}
        \trivCof_{n} = \diskInc_{n+1} \circ \iCptd_{n+1}(\sigma_n) \colon
        \iCptd_{n+1}(\Disk^n) \to \Equiv^{n+1}
    \end{equation}
    which is used to characterise equifibrations by Fujii et
    al.~\cite{fujii$o$equifibrationsStrictWeak2025}.
\end{example}

\begin{example}\label{ex:walking-triangle}
    The \emph{walking commutative triangle} \(\Triangle\) is the following
    \(2\)\=/computad with invertible generators:
    \begin{align}\label{eq:walking-triangle}
        \begin{aligned}
            \Triangle_{0}
                &= \set{x,y,z} \\
            \Triangle_{1}
                &= (\Triangle_{0},\set{f,g,h},\emptyset,\phi^{\Triangle}_{1},!)
            \\
            \Triangle
                &= (\Triangle_{1},\emptyset,\set{e},!,\psi^{\Triangle}_{2})
        \end{aligned} &&
        \begin{aligned}
            \phi^{\Triangle}_{1}(f)
            &= (\gen_{0}(x),\gen_{0}(y)) \\
            \phi^{\Triangle}_{1}(g)
            &= (\gen_{0}(y),\gen_{0}(z)) \\
            \phi^{\Triangle}_{1}(h)
            &= (\gen_{0}(x),\gen_{0}(z)) \\
            \psi^{\Triangle}_{2}(e)
            & = (\gen_{1}(f)\compCell^{2}_{1}\gen_{1}(g),\gen_{1}(h))
        \end{aligned}
    \end{align}
    We may illustate this computad as follows, where \(\sim\) denotes
    invertibility:
    \begin{equation}\label{eq:walking-triangle-fig}
        \begin{tikzcd}
            & x\ar[rd]{}{g}\ar[d,Rightarrow,"e"{swap},"\sim"{sloped}] & \\
            y\ar[rr,swap]{}{h}\ar[ru]{}{f} & \phantom{a} & z
        \end{tikzcd}
    \end{equation}
    The name \emph{walking commutative triangle} refers to the fact that morphisms
    out of \(\Triangle\) correspond precisely to triangles commuting up to
    equivalence.
\end{example}

\begin{example}\label{ex:walking-cylinders}
    Benjamin et al.~\cite{benjaminNaturalityHigherdimensionalPath2025} have
    given a description of a family of computad corresponding to \emph{directed
    cylinders} using their naturality construction. Modifying their construction
    by making certain generators invertible can be used to construct
    \emph{reversible cylinders}. In analogy to the case of strict
    \(\omega\)\=/categories~\cite{lafontFolkModelStructure2010}, reversible
    cylinders are expected to play a crucial role in the definition of a model
    structure for weak \(\omega\)\=/categories. This has been a longstanding
    conjecture, which we discuss further in Section~\ref{subsec:invertibility}.
    \begin{align}\label{eq:square-cylinder}
        \begin{tikzcd}[ampersand replacement=\&]
            \bullet \& \bullet \\ \bullet \& \bullet
            \arrow[from=1-1,to=1-2]
            \arrow[from=2-1,to=2-2]
            \arrow["\sim"{sloped}, swap, from=1-1,to=2-1]
            \arrow["\sim"{sloped}, from=1-2,to=2-2]
            \arrow["\sim"{sloped}, Rightarrow, from=1-2, to=2-1]
        \end{tikzcd} &&&&
        \begin{tikzcd}[ampersand replacement=\&]
            \bullet \& \bullet \&\& \bullet \& \bullet \\
            \bullet \& \bullet \&\& \bullet \& \bullet
            \arrow[from=1-1,to=1-2, bend left = 30]
            \arrow[""{name=src0}, from=2-1,to=2-2, bend left = 30]
            \arrow[""{name=tgt0}, from=2-1,to=2-2, bend right = 30]
            \arrow["\sim"{sloped}, swap, from=1-1,to=2-1]
            \arrow[""{name=src2}, "\sim"{sloped}, from=1-2,to=2-2]
            \arrow["\sim"{sloped}, Rightarrow, bend right = 20, from=1-2,
                to=2-1, shorten >= .75em, shorten <= .75em]
            \arrow[Rightarrow, from=src0, to=tgt0, shorten <= .4em]
            \arrow[""{name=src1},from=1-4,to=1-5, bend left = 30]
            \arrow[""{name=tgt1}, from=1-4,to=1-5, bend right = 30]
            \arrow[from=2-4,to=2-5, bend right = 30]
            \arrow[""{name=tgt2}, "\sim"{sloped}, swap, from=1-4,to=2-4]
            \arrow["\sim"{sloped}, from=1-5,to=2-5]
            \arrow["\sim"{sloped}, Rightarrow, bend left = 20, from=1-5,
                to=2-4, shorten >= .75em, shorten <= .75em]
            \arrow[Rightarrow, from=src1, to=tgt1, shorten <= .4em]
            \arrow["\sim"{sloped}, Rightarrow, from=src2, to=tgt2,
                shorten <= 1.5em, shorten >= 1.5em]
            \arrow[no head, from=src2, to=tgt2,
                shorten <= 1.5em, shorten >= 1.5em]
        \end{tikzcd}
    \end{align}
\end{example}

\begin{example}\label{ex:invertible-gsets}
    For the definition of computads with invertible generators, we defined an
    inclusion of globular sets into computads by making every cell into an
    ordinary generator. Similarly, we could define out of an \(n\)\=/globular
    set \(X\) an \(n\)\=/computad \(\invCptd_n(X)\) by making the cells of \(X\)
    into invertible generators. More precisely, we define
    \(\invCptd_n(X) = \iCptd_n(X)\) for \(n\le 0\) and then we let
    \begin{equation}\label{eq:invertible-gsets}
        \invCptd_{n+1}(X) =
        (\invCptd_n(\tr^{n+1}_n X),\emptyset,X_{n+1},!,\psi_{n+1}^{\invCptd(X)})
    \end{equation}
    where the attaching function \(\psi_{n+1}^{\invCptd(X)}\) sends an
    \((n+1)\)\=/cell \(f \colon x \to y\) to
    \begin{equation}\label{eq:invertible-gsets-bdry}
        \psi^{\invCptd(X)}_{n+1}(f) = \begin{cases}
            (x,y) & \text{if }n=0\\
            (\fwd_{n}(\igen_{n} (x)),\fwd_{n}(\igen_{n}(y))) & \text{otherwise}
        \end{cases}
    \end{equation}
    We will see in Proposition~\ref{cor:all-gen-invertible-is-groupoid} that
    this construction generates a weak \(\omega\)\=/category whose cells are
    all invertible.
\end{example}

\begin{proposition}\label{prop:tr-is-rali}
    The truncation functor \(\tr^m_n\colon \iComp_m\to \iComp_n\) is a right
    adjoint left inverse for all \(-1\le n < m \le \omega\). In particular,
    every \(n\)\=/computad is canonically an \(m\)\=/computad.
\end{proposition}
\begin{proof}
    We construct the left adjoint and the counit of the adjunction:
    \begin{align}\label{eq:counit-sk}
        \sk^m_n&\colon \iComp_n\to \iComp_m &
        \counitSk^m_n&\colon \sk^m_n\tr^m_n \Rightarrow \id
    \end{align}
    and prove that it \(\sk^m_n\) is a right inverse to \(\tr^m_n\), and that
    \(\counitSk^m_n\) with the identity natural transformation satisfy the
    triangle equations:
    \begin{align}\label{eq:counit-sk-equations}
        \tr^m_n\sk^m_n &= \id &
        \tr^m_n\counitSk^m_n &= \id &
        \counitSk^m_n\sk^m_n &= \id
    \end{align}
    We will do this by recursion on \(0\le k \le \omega\) such that
    \(m = n + 1 + k\).

    In the base case \(k = 0\) and \(n = -1\), the skeleton
    functor picks the empty set \(\emptyset\) and \(\counitSk^0_{-1}\) picks the
    unique morphism out of it. In the base case where \(k=0\) and \(n\in\Nat\),
    the left adjoint and the counit are defined by
    \begin{align}\label{eq:sk-counit-base}
        \sk^{n+1}_n(C) &= (C, \emptyset,\emptyset,!,!) &
        \sk^{n+1}_n(f) &= (f, !, !) &
        \counitSk^{n+1}_{n,C} &= (\id_{\tr_n^{n+1} C} , ! , !)
    \end{align}
    where \(!\) denotes the unique morphism out of the empty set. In this case,
    we can check Equations~\ref{eq:counit-sk-equations} directly.

    For the case where \(k\) is a successor, we use that right adjoints are
    closed under composition. More precisely,
    \(\tr^{n+1+k}_n = \tr^{n+k}_n\tr^{n+1+k}_{n+k}\) by
    Property~\ref{item:prop-cocycle}, so we define:
    \begin{equation}\label{eq:sk-counit-step}
        \begin{aligned}
            \sk^{n+1+k}_n
                &= \sk^{n+1+k}_{n+k}\sk^{n+k}_n \\
            \counitSk^{n+1+k}_n
                &= \counitSk^{n+1+k}_{n+k}\circ
                (\sk^{n+1+k}_{n+k}\counitSk^{n+k}_n\tr^{n+1+k}_{n+k})
        \end{aligned}
    \end{equation}
    Equations~\ref{eq:counit-sk-equations} in this case follow from the same
    equations for \(k-1\) and for \(0\).

    Finally, for the case that \(k = m = \omega\) we define the left adjoint and
    the counit to be the unique functor and natural transformation with
    components
    \begin{align}\label{eq:sk-limit}
        \tr^\omega_r\sk^\omega_n &= \begin{cases}
            \tr^n_r &\text{if }r<n \\
            \id &\text{if }r=n \\
            \sk^r_n &\text{if }r>n
        \end{cases} &
        \tr^\omega_r\counitSk^\omega_n &= \begin{cases}
            \id &\text{if }r<n \\
            \id &\text{if }r=n \\
            \counitSk^r_n &\text{if }r>n
        \end{cases}
    \end{align}
    These are well-defined by Property~\ref{item:prop-cocycle} and by the
    following generalisation of Equation~\ref{eq:counit-sk-equations}:
    \begin{align}\label{eq:counit-sk-equations-general}
        \tr^m_r \sk^m_n = \begin{cases}
            \tr^n_r &\text{if }r<n \\
            \id &\text{if }r=n \\
            \sk^r_n &\text{if }r>n
        \end{cases} &&
        \tr^m_r\counitSk^m_n &= \begin{cases}
            \id &\text{if }r<n \\
            \id &\text{if }r=n \\
            \counitSk^r_n &\text{if }r>n
        \end{cases}
    \end{align}
    which can be proven inductively on \(k\) using the definitions.
\end{proof}

\noindent
Since the skeleton functors \(\sk^m_n\) are left adjoints left inverses, they
are fully faithful. Using these functors, we will implicitly identify
\(n\)\=/computads with \(m\)\=/computads that have no generators above dimension
\(n\).

\begin{proposition}\label{prop:representability-cell-sphere-iset}
    For every computad \(C\), there exist natural isomrophisms
    \begin{equation}\label{eq:representability-cell-sphere-iset}
        \begin{aligned}
            \iCell(C)_{n} &\cong \iComp(\iCptd(\Disk^n),C) \\
            \Par_{n}(\iCell(C)) &\cong \iComp(\iCptd(\Sphere^{n}),C) \\
            \iSet_{n+1}(C) &\cong \iComp(\Equiv^{n+1},C)
        \end{aligned}
    \end{equation}
    Under these isomorphisms, composition with \(\iCptd(\sphereInc_n)\) sends a
    cell to its boundary, while composition with \(\diskInc_{n+1}\)
    sends an element \(i\in\iSet_{n+1}(C)\) to the cell \(\fwd_{n+1}(i)\).
    Composition with \(\equivInc_{n+1}\) sends \(i\in\iSet_{n+1}(C)\) to its
    boundary.
\end{proposition}
\begin{proof}
    The first two isomorphisms follow form the adjunction \(\iCptd\dashv\iCell\)
    and isomorphisms~\eqref{eq:disk-sphere-representable}. Using the adjunction
    of Proposition~\ref{prop:tr-is-rali} and the second isomorphism, we see
    that morphisms \(f\colon \Equiv^{n+1}\to C\) are in bijection with a
    parallel pair \(A\in \Par_n(\Cell(C))\) and an element
    \(i\in\iSet_{n+1}(C)\) -- the image of the invertible generator \(c_{n+1}\)
    -- satisfying the compatibility condition:
    \begin{equation}\label{eq:repr-iset-bdry}
        \ibdry_{n+1,C}(i) = A
    \end{equation}
    which corresponds to commutativity of the second square
    of~\eqref{eq:def-morphism}. The compatibility condition completely
    determines the parallel pair \(A\), so morphisms \(\Equiv^{n+1}\to C\) are
    in natural bijection to elements of \(\iSet_{n+1}(C)\).
\end{proof}



\section{Weak \texorpdfstring{\(\omega\)\=/}{ω-}categories}
\label{sec:omega-cats}

In this section, we introduce the category of ordinary computads, in the sense
of Dean et al.~\cite{deanComputadsWeak$omega$categories2024} and show that it is
a coreflective subcategory of \(\iComp\). We then show that both the inclusion
and the coreflector preserve the generated weak \(\omega\)\=/categories, so that
kinds of computads give rise to cofibrant weak \(\omega\)\=/categories, in the
sense of Batanin and Leinster~\cite{leinsterHigherOperadsHigher2003}. We then
study coinductively invertible cells and describe the \(\omega\)\=/categories
generated by computads as certain pushouts.

\subsection{Comparison with ordinary computads}
\label{sec:ord-computads}

We first introduce the category of ordinary computads and show that it coincides
with the one defined by Dean et
al.~\cite{deanComputadsWeak$omega$categories2024}. We then show that the
inclusion of ordinary computads to computads with invertible generators admits
a right adjoint.

\begin{definition}\label{}
    The category \(\Comp_n\) of \emph{\(n\)\=/computads} for
    \(-1\le n \le \omega\) is the full subcategory of \(\iComp_n\) defined
    recursively by:
    \begin{itemize}
        \item \(\Comp_n = \iComp_n\) for \(n \le 0\),
        \item \(C\in \Comp_{n+1}\) for \(n\in\Nat\) when
            \(\tr^{n+1}_{n} C \in \Comp_{n}\) and \(U_{n+1}^C = \emptyset\)
        \item \(C\in \Comp_\omega\) when \(\tr_n C \in \Comp_n\)
            for all \(n < \omega\)
    \end{itemize}
    In other words, \(\Comp_n\) consists of the computads that have no
    invertible generators. The truncation and skeleton functors restrict to the
    categories \(\Comp_n\) by definition.
\end{definition}

The left adjoint \(\iCptd\) factors through the category of \(n\)\=/computads,
and we will denote the restriction by \(\Cptd_n\colon\gSet_n\to\Comp_n\). We
also denote the restriction of the cell functor by
\(\Cell_n\colon\Comp_n\to\gSet_n\). Since \(\Comp_n\) is a full subcategory,
these functors remain adjoints to each other with the same unit and the
restriction of the counit.

\begin{proposition}\label{prop:dean-comparision}
    The category \(\Comp_n\) is isomoprhic to the category of computads of Dean
    et al.~\cite{deanComputadsWeak$omega$categories2024}. Under this
    isomorphism, the adjunction \(\Cptd\dashv\Cell\) agrees with their
    adjunction \(\operatorname{Free}\dashv \Cell\).
\end{proposition}
\begin{proof}
    The statement can be proven by induction on \(-1\le n \le \omega\),
    together with the fact that \(\iSet_n(C) = \emptyset\) for every ordinary
    \(n\)\=/computad \(C\) and that the set \(\Par_n(\Cell_n(C))\) coincides
    with their functor \(\operatorname{Sphere}_n(C)\).
\end{proof}

\begin{corollary}\label{cor:leinster-comparision}
  The category of algebras for the monad \(T\) induced by the adjunction
  \(\iCptd\dashv \iCell\) is the category of weak \(\omega\)\=/categories and
  strict \(\omega\)\=/functors of
  Leinster~\cite{leinsterHigherOperadsHigher2003}.
\end{corollary}
\begin{proof}
  The monad on \(\gSet\) induced by the adjunction with \(\iComp\) agrees with
  that induced by the adjunction with \(\Comp\). In light of
  Proposition~\ref{prop:dean-comparision}, algebras of this monad are
  Leinster's weak
  \(\omega\)\=/categories~\cite[Corollary~7.36]{
    deanComputadsWeak$omega$categories2024}
\end{proof}

\noindent
By definition of the category \(\Cat_\omega\) of weak \(\omega\)\=/categories as
a category of algebras for a monad, it is equipped with a free-forgetful
adjunction to the category \(\gSet\) of gloublar sets, as well as morphisms of
adjunctions:
\begin{equation}\label{eq:summary-of-adjunctions}
    \begin{tikzcd}[column sep=8em,row sep=huge]
        {\Comp} & {\iComp} & {\Cat_\omega} \\
        & \gSet
        \arrow[hook, from=1-1, to=1-2]
        \arrow["\Komp", bend left =15, from=1-1, to=1-3]
        \arrow["\iKomp", from=1-2, to=1-3]
        \arrow[""{name=0, anchor=center, inner sep=0}, "\Cell"',
            shift right=2, from=1-1, to=2-2]
        \arrow[""{name=1, anchor=center, inner sep=0}, "\Cptd"',
            shift right=2, from=2-2, to=1-1]
        \arrow["\dashv"{anchor=center, rotate=-118}, draw=none, from=1, to=0]
        \arrow[""{name=2, anchor=center, inner sep=0}, "\iCell"',
            shift right=2, from=1-2, to=2-2]
        \arrow[""{name=3, anchor=center, inner sep=0}, "\iCptd"',
            shift right=2, from=2-2, to=1-2]
        \arrow["\dashv"{anchor=center, rotate=-180}, draw=none, from=3, to=2]
        \arrow[""{name=4, anchor=center, inner sep=0}, "F"',
            shift right=2, from=2-2, to=1-3]
        \arrow[""{name=5, anchor=center, inner sep=0}, "U"',
            shift right=2, from=1-3, to=2-2]
        \arrow["\dashv"{anchor=center, rotate=118}, draw=none, from=4, to=5]
    \end{tikzcd}
\end{equation}
The composition and identity operations of computads extend to operations in
arbitrary weak \(\omega\)\=/categories as well: using the adjunction
\(F\dashv U\), the Yoneda lemma, and the pushout
square~\eqref{eq:binary-trees-po}, we get a pair of isomorphisms:
\begin{equation}\label{eq:disk-repr-cells}
    \begin{aligned}
        \chi_\bullet
            &\colon X_n\xlongrightarrow{\sim}\Cat_\omega(F\Pos(D_n), X) \\
        \chi_{\bullet,\bullet}
            &\colon X_n\times_{X_k}X_n
            \xlongrightarrow{\sim}\Cat_\omega(F\Pos(P^k_n), X)
    \end{aligned}
\end{equation}
These allows us to define the identity of a cell \(x\in X_n\) in a weak
\(\omega\)\=/category \(X\), and the \(k\)\=/composite of a pair of cells
\(x,y\in X_n\) such that \(\tgt_k(x) = \src_k(y)\) by letting:
\begin{align}\label{eq:id-comp-category}
    \idCell_{n+1}(x) &= \chi_{x}(\idCell_{n+1}(\gen_{n+1}(c_n))) &
    x \compCell^{n}_{k} y &= \chi_{x,y}(\gen_n(a) \compCell^{n}_{k} \gen_n(b))
\end{align}
where \(c_n\) is the top-dimensional position of the tree \(D_n\) and \(a,b\)
are the top-dimensional cells of \(P^k_n\). When \(X = \iKomp C\) is freely
generated by a computad, this definition of composition and identities restricts
to the one defined in Section~\ref{sec:definition}.

\begin{theorem}\label{thm:unfolding}
    The inclusion \(\Comp\hookrightarrow\iComp\) admits a right adjoint
    \(\Unf\).
\end{theorem}

\begin{proof}
    We will construct recursively on \(-1\le n\le \omega\) and for every
    \(n\)\=/computad with invertible generators \(C\), an ordinary
    \(n\)\=/computad \(\Unf_nC\) together with a morphism
    \(\counitUnf_{n,C}\colon \Unf_n C\to C\) such that postcomposition
    defines a natural bijection:
    \begin{equation}\label{eq:unf-is-adjoint}
        (\counitUnf_{n,C})_* \colon
            \Comp_n(D,\Unf_n C)\xlongrightarrow{\sim}\iComp_n(D,C)
    \end{equation}
    for every ordinary computad \(D\in\Comp_n\). Moreover, we will construct
    them such that they commute with the truncation functors in the obvious
    sense.

    In the base cases \(n \le 0\) the categories \(\Comp_n\) and \(\iComp_n\)
    agree, so we may define \(\Unf_nC = C\) and \(\counitUnf_{n,C} = \id_C\).
    Suppose therefore that the structures above have been defined for some
    \(n\in\Nat\). Then specialising the isomorphism to the \(n\)\=/computads
    \(D = \Cptd_n\Disk^k\) for \(k\le n\) and using the Yoneda lemma and the
    adjunction \(\Cptd\dashv\Cell\), we obtain an isomorphism of
    \(n\)\=/globular sets of the form:
    \begin{equation}\label{eq:unf-preserves-cells}
       \iCell_n(\counitUnf_{n,C}) \colon
            \Cell_n(\Unf_n C)\xlongrightarrow{\sim}\iCell_n(C)
    \end{equation}
    We fix now an \((n+1)\)\=/computad with invertible generators \(C\). The
    ordinary computad \(\Unf_{n+1}C = (\Unf_n C_n, V_{n+1}^{\Unf C}, \emptyset,
    \phi_{n+1}^{\Unf C}, !)\) is obtained by turning all cells of \(C\) obtained
    from \(\iSet_{n+1}(C)\) into ordinary generators, and the morphism
    \(\counitUnf_{n+1,C}\)
    sends the generators to the cells they correspond to. More precisely, we
    define the set of generators to be the coproduct:
    \begin{equation}\label{eq:unf-generators}
        V_{n+1}^{RC}
                = V_{n+1}^C \amalg \iSet_{n+1}(C) \amalg \iSet_{n+1}(C) \amalg \iSet_{n+1}(C)
    \end{equation}
    and we define a function \(\counitUnf_{n+1,C,V}\) out of it using the
    universal property:
    \begin{equation}\label{eq:counitUnf-gen}
        \begin{aligned}
            \counitUnf_{n+1,C,V}
                &\colon V_{n+1}^{RC} \to \iCell_{n+1}(C)_{n+1}\\
            \counitUnf_{n+1,C,V}
                &= \copair{\gen_{n+1}^C,\fwd_n,\linv_{n+1},\rinv_{n+1}}
        \end{aligned}
    \end{equation}
    We then define the attaching function \(\phi_{n+1}^{\Unf C}\) to be the
    unique function that fits in the following square:
    \begin{equation}\label{eq:unf-attaching}
        \begin{tikzcd}[ampersand replacement=\&, column sep = large]
            {V_{n+1}^{RC}} \& {\Par_n(\Cell_n(R_nC_n))} \\
            {\iCell_{n+1}(C)_{n+1}} \& {\Par_n(\iCell_n(C_n))}
            \arrow["{\phi_{n+1}^{RC}}", dashed, from=1-1, to=1-2]
            \arrow["{\counitUnf_{n+1,C,V}}"',
                from=1-1, to=2-1]
            \arrow["\Par_n(\iCell_n(\counitUnf_{n,C}))",
                "\sim"{rotate=90, anchor = south},
                 from=1-2, to=2-2]
            \arrow["{\bdry_n^C}"', from=2-1, to=2-2]
        \end{tikzcd}
    \end{equation}
    Commutativity of this square implies that the counit:
    \begin{equation}\label{eq:counitUnf}
        \counitUnf_{n+1,C} =
            (\counitUnf_{n,C_n},\counitUnf_{n+1,C,V},!)\colon \Unf_{n+1}C\to C
    \end{equation}
    is a well-defined morphism. We see immediately that the structures we built
    commute with the truncation functors. To complete the inductive step, we
    need to show that every morphism \(f\colon D\to C\) from an ordinary
    computad \(D\) to \(C\) factors uniquely through \(\counitUnf_{n+1,C}\):
    \begin{equation}\label{eq:counitUnf-lifting}
        \begin{tikzcd}
            & {\Unf_{n+1}C} \\
            D & C
            \arrow["{\counitUnf_{n+1,C}}", from=1-2, to=2-2]
            \arrow["{f^\dagger}", dashed, from=2-1, to=1-2]
            \arrow["f"', from=2-1, to=2-2]
        \end{tikzcd}
    \end{equation}
    We will construct the morphism \(f^\dagger\) by structural induction.

    Suppose first that \(D = \Cptd_{n+1}\Disk^{n+1}\) so that \(f\) clasifies
    some cell \(c\in\iCell_{n+1}(C)_{n+1}\). Unless \(c\) is a coherence
    cell, there exists unique generator \(v\in V_{n+1}^{\Unf C}\) such that the
    cell \(c' = \gen_{n+1}(v)\) satisfies:
    \begin{equation}\label{eq:counitUnf-cell}
        \iCell_{n+1}(\counitUnf_{n+1,C})(c') = c
    \end{equation}
    Moreover, every other top-dimensional cell of \(\Unf_{n+1}C\) is a coherence
    cell, so it can not satisfy this equation. Therefore, the morphism
    \(f^\dagger\) clasifying \(c'\) is the unique one
    making triangle~\eqref{eq:counitUnf-lifting} commute.

    Suppose then that \(c = \coh_{B,A}[g]\) is a coherence cell. By structural
    recursion, we may assume that there exists unique \(g^\dagger\) such that
    \(g = \counitUnf_{n+1,C}g^\dagger\). Since \(\varepsilon_{n+1,C}\) does not
    send any generator to a coherence cell, it follows that
    \(c' = \coh_{B,A}[g^\dagger]\) is the unique cell satisfying
    Equation~\ref{eq:counitUnf-cell} and hence the morphism \(f^\dagger\)
    tha classifies \(c'\) is again the unique one making
    triangle~\eqref{eq:counitUnf-lifting} commute.

    Let now \(D\) be an arbitrary ordinary computad. Then by structural
    recursion, for every \(v\in V_{n+1}^D\), there exists unique cell
    \(f^\dagger_V(v)\in\iCell_{n+1}(\Unf_{n+1}C)_{n+1}\) such that
    \begin{equation}\label{eq:counitUnf-mor-cell}
        \iCell_{n+1}(\counitUnf_{n+1,C})(f^\dagger_V(v)) = f_V(v)
    \end{equation}
    Moreover, there exists unique
    \(f_n^\dagger\colon D_n\to \Unf_nC_n\) satisfying
    \(\varepsilon_{n,C_n}f_n^\dagger = f_n\) by the recursion on dimension.
    These are related by the equation:
    \begin{equation}\label{eq:counitUnf-mor-compat}
        \bdry_{n,\Unf_{n+1}C} (f^\dagger_V(v))
            =
        \Par_n(\iCell_n(f_n))(\phi_{n+1}^D(v))
    \end{equation}
    Indeed, by construction, both sides of the equation become equal after
    applying the isomorphism \(\Par_n(\iCell_n(\varepsilon_{n,C_n}))\), so they
    must already be equal. It follows that
    \begin{equation}\label{eq:Unf-transpose}
        f^\dagger = (f^\dagger_n, f^\dagger_V, !) \colon D \to \Unf_{n+1} C
    \end{equation}
    is a well-defined morphism. By definition of \(f^\dagger_n\) and
    \(f^\dagger_V\), this is the unique morphism that makes
    triangle~\eqref{eq:counitUnf-lifting} commute. This concludes the structural
    recursion, and the successor step of the recursion on dimension.

    Suppose finally that the structures above have been defined for all
    \(n\in\Nat\) and they commute with the truncation functors. Then for
    \(n = \omega\), we may define for every computad with invertible generators
    \(C\), the ordinary computad \(\Unf C\) to be the one consisting of
    \(\Unf_n C_n\) for every \(n < \omega\), and we define
    \(\counitUnf_{C}\) to be the morphism consisting of
    \(\counitUnf_{n,C_n}\) for \(n < \omega\). By the definition of the
    categories of computads as limits, we have that:
    \begin{equation}
        \begin{tikzcd}[column sep = huge]
            {\iComp(D,\Unf C)} & {\Comp(D,C)} \\
            {\lim\limits_{n\in\Nat}\iComp_n(D_n,\Unf_nC_n)} &
            {\lim\limits_{n\in\Nat}\Comp_n(D_n,C_n)}
            \arrow["{(\counitUnf_C)_*}", from=1-1, to=1-2]
            \arrow[equals, from=1-1, to=2-1]
            \arrow[equals, from=1-2, to=2-2]
            \arrow["{(\counitUnf_{n,C_n})_*}"', "\sim", from=2-1, to=2-2]
        \end{tikzcd}
    \end{equation}
    Since postcomposition with \(\counitUnf_{n,C_n}\) is an isomorphism for all
    \(n\in\Nat\), the top row must be an isomorphism as well, finishing the
    induction.
\end{proof}

\begin{corollary}\label{cor:unf-preserves-Komp}
    The weak \(\omega\)\=/category generated by a computad \(C\) with invertible
    generators is isomorphic to the one generated by the ordinary computad
    \(\Unf C\).
\end{corollary}
\begin{proof}
    It suffices to show that the counit of the adjunction
    \(\counitUnf_C\colon \Unf C\to C\) gives rise to an isomorphism of
    weak \(\omega\)\=/categories:
    \begin{equation}\label{eq:unf-preserves-Komp}
        \iKomp(\counitUnf_C)\colon \Komp(\Unf C)\to \iKomp(C)
    \end{equation}
    This morphism become the isomorphism~\eqref{eq:unf-preserves-cells} after
    applying the forgetful functor \(U\colon \Cat_\omega\to \gSet\). The
    forgetful functor is strict monadic, so it reflects isomorphisms. Therefore,
    \(\iKomp(\counitUnf_C)\) is already an isomorphism of
    \(\omega\)\=/categories.
\end{proof}

\noindent
Since the category \(\Cat_\omega\) of weak \(\omega\)\=/categories is locally
finitely
presentable~\cite[Corollary~4.5]{deanComputadsWeak$omega$categories2024}, the
set of morphisms
\begin{equation}\label{eq:gen-cofibrations}
    I = \set{ F\sphereInc_n \colon F\Sphere^{n-1} \to F\Disk^n \mid n\in\Nat}
\end{equation}
cofibrantly generated a weak factorisation system by the small object argument.
Morphsims in the left class of this system are called \emph{cofibrations}, while
morphisms in the right class are called \emph{trivial fibrations}. We call a
weak \(\omega\)\=/category \(X\) \emph{cofibrant} when the unique morphism
\(\emptyset\to X\) from the initial \(\omega\)\=/category is a cofibration.
This corresponds precisely to being generated by an ordinary computad, as shown
by Markakis~\cite[Proposition~19 \& Corollary~22]{
markakisComputadsGeneralisedSignatures2024}. Combining this result with
Corollary~\ref{cor:unf-preserves-Komp}, we get the following result:

\begin{corollary}\label{cor:computads-cofibrant}
    Weak \(\omega\)\=/categories generated by a computad with invertible
    generators are cofibrant.
\end{corollary}

\begin{remark}\label{rmk:finiteness}
    While weak \(\omega\)\=/categories generated by computads with invertible
    generators are isomorphic to oneew generated by ordinary computad, our
    framework gives a powerful systematic way to handle the complexity arising
    for freely inverting generators. For instance, regarding
    Example~\ref{ex:walking-equivalence}, the computad with invertibile
    generators \(\Equiv^{n+1}\) enjoys a simple and concise description. On the
    other hand, in order to present the same \(\omega\)\=/category as an
    ordinary computad, one would need to give a direct definition of
    \(\Unf\Equiv^{n+1}\). This has been done by Ozornova and
    Rovelli~\cite{ozornovaWhatEquivalenceHigher2024} in the strict case, and by
    Fujii et al.~\cite{fujii$o$equifibrationsStrictWeak2025} in the weak case.
    Nonetheless, their definition is significantly more involved. Our framework
    similarly allows for simple descriptions of other \(\omega\)\=/categories as
    in Examples~\ref{ex:walking-triangle} and~\ref{ex:walking-cylinders}.
    Moreover, we believe this framework could be useful in describing the
    localisation of \(\omega\)\=/categories obtained by freely inverting cells
    above a certain dimension \(n\).
\end{remark}

\subsection{Invertibility}
\label{subsec:invertibility}

We recall the notion of \emph{invertible cells} in \(\omega\)\=/categories and
proceed to show that invertible generators are indeed invertible. We then
proceed to show that \(\omega\)\=/categories generated by computads with
invertible generators can be constructed by a sequence of pushouts.

\begin{definition}\label{def:invertibility-structure}
  The set \(\Inv_{X,f}\) of \emph{invertibility structures} on a
  cell \(f\colon x\to y\) of positive dimension in an \(\omega\)\=/category
  \(X\) is defined coinductively by the following destructors, producing out
  of an invertibility structure \(i\in \Inv_{X,f}\):
  \begin{itemize}
  \item a cell \(\linv(i)\colon y\to x\), called a \emph{left inverse} of \(f\)
  \item a cell \(\rinv(i)\colon y\to x\), called a \emph{right inverse} of \(f\)
  \item a cell \(\lunit(i)\colon \linv(i)\compCell_n f \to \idCell_{n+1}(y) \),
    called a \emph{left unit} of \(f\)
  \item a cell \(\runit(i)\colon f\compCell_n\rinv(i)\to \idCell_{n+1}(y)\),
    called a \emph{right unit} of \(f\),
  \item invertibility structures \(\ilunit(i)\) on \(\lunit(i)\), and
    \(\irunit(i)\) on \(\runit(i)\)
  \end{itemize}
  We say that \(f\) is \emph{invertible} when there exists an invertibility
  structure on it.
\end{definition}

In other words, the family of invertibility structures on cells of an
\(\omega\)\=/category \(X\) is the carrier of the terminal coalgebra for the
endofunctor \(G_X\) on \({\Set\downarrow X^+}\) taking a family
\(W\) indexed by positive-dimensional cells of \(X\) to the family whose
component at a cell \(f \colon u\to v\) is given by:
\begin{equation}
\label{eq:invertibility-endofunctor}
    G_X(W)_f =
    \left(\coprod_{f_{l} \colon v \to u}\ \coprod_{f_{L} \colon f_{l}\ast f
    \to \idCell }W_{f_{L}}\right) \times \left(\coprod_{f_{r} \colon v \to u}\
    \coprod_{f_{R} \colon f\ast f_{r} \to \idCell }W_{f_{R}}\right)
\end{equation}
Existence of the terminal coalgebra is guaranted by Adamek's
theorem~\cite{adamekFreeAlgebrasAutomata1974} and preservation of connected
limits by the endofunctor. The latter follows by commutativity of connected
limits with coproducts in \(\Set\) and the computation of limits in slice
categories.

\begin{remark}\label{rmk:2sided-invertibility}
  We note that this definition of invertibility a priori differs from the
  one where the left and right inverse are assumed to be strictly equal.
  Nonetheless, it has been shown by
  Rice~\cite{riceCoinductiveInvertibilityHigher2020} and later by Fujii et
  al.~\cite[Proposition~3.3.1]{fujii$o$equifibrationsStrictWeak2025} that the
  two-notions coincide. Our use of this variant of invertibility is motivated by
  the work of Hadzihasanovic et
  al.~\cite{hadzihasanovicModelCoherentWalking2024} that show that the strict
  version of \(\Equiv^{n+1}\) is contractible, and the work of Fujii et
  al.~\cite[Theorem~5.4.9]{fujii$o$equifibrationsStrictWeak2025}, who show that
  the right class of the weak factorisation system cofibrantly generated by the
  set
  \begin{equation}\label{eq:gen-trivial-cofibrations}
    J = \set{ K\trivCof_n \colon F\Disk^{n} \to K\Equiv^{n+1} \mid n \in\Nat}
  \end{equation}
  is that of \emph{equifibrations}: morphisms \(f\colon X\to Y\) satisfying that
  for every invertible \(c\colon y\to y'\) in \(Y\) and every \(x\in X\) such
  that \(f(x) = y\), there exists an invertible \(c'\colon x\to x'\) such that
  \(f(c') = c\).
\end{remark}

\begin{proposition}\label{prop:inv-gens-are-inv}
    In a computad \(C\) with invertible generators,
    cells of the form \(\fwd_{n+1}(i)\) for \(i\in \iSet_{n+1}(C)\)
    are invertible.
\end{proposition}
\begin{proof}
  Consider the \(G_{\iKomp C}\)\=/coalgebra given by:
  \begin{equation}\label{eq:inv-gens-are-inv-family}
    \begin{gathered}
      I_{C,f} = \set{i\in \iSet_{n+1}(C) \such \fwd i = f} \\
      i\mapsto((\linv i,\fwd\lunit i,\lunit i),(\rinv i,\fwd\runit i,\runit i))
    \end{gathered}
  \end{equation}
  By definition of \(\Inv_{\iKomp C}\) as a terminal coalgebra, there exists a
  unique coalgebra morphism \(\inv_C\colon I_C\to \Inv_{\iKomp C}\), showing
  every cell of the form \(\fwd_{n+1}(i)\) is invertible.
\end{proof}

\begin{corollary}\label{cor:all-gen-invertible-is-groupoid}
    In a computad that has no
    positive-dimensional ordinary generators, every cell is invertible.
\end{corollary}
\begin{proof}
  We construct an invertibility structure on every cell by structural recursion.
  For cells of the form \(\fwd_{n+1}(i)\) we use the one constructed in
  Proposition~\ref{prop:inv-gens-are-inv}. For their inverses
  \(\rinv_{n+1}(i)\) and \(\linv_{n+1}(i)\) such a structure has been
  constructed by Fujii et al.~\cite[Proposition~3.3.2]{
  fujiiWeaklyInvertibleCells2023}. Finally, for cells of the form
  \(\coh_{B,A}[f]\), we may assume that an invertibility structure is given on
  \(f_V(p)\) for every position \(p\in \Pos_{n+1}(B)\). Out of these structures,
  Benjamin and Markakis~\cite{benjaminInvertibleCells$omega$categories2024}
  construct an invertibility structure on \(\coh_{B,A}[f]\).
\end{proof}

Morphisms of \(\omega\)\=/categories \(f\colon X\to Y\) preserve invertibility.
Such a morphism gives rise to a morphism of endofunctors
\((\Sigma_f,\alpha_f)\colon G_X\to G_Y\) consisting of the reindexing functor:
\begin{equation}\label{eq:Inv-functorial-reindexing}
\begin{aligned}
    \Sigma_f&\colon \Set\downarrow X^+
        \to \Set\downarrow Y^+ &&&
    (\Sigma_fW)_y &= \coprod_{f(x) = y} W_x
\end{aligned}
\end{equation}
and the natural transformation
\(
\alpha_f\colon \Sigma_f G_X \Rightarrow G_Y \Sigma_f
\)
whose component at a family \(W\) over \(X^+\) and some \(y\in Y_{n+1}\) is
given by:
\begin{equation}\label{eq:Inv-functorial-reindexing-transformation}
  \begin{split}
    \alpha_{f,W,y}(x,&((x_l,x_L,w_L),(x_r,x_R,w_R))) \\
      &= ((f(x_l),f(x_L),(x_L,w_L)),(f(x_r),f(x_R),(x_R,w_R)))
  \end{split}
\end{equation}
This is well-defined since \(f\) preserves composition and identities. This
morphism defines a functor at the level of coalgebras:
\begin{equation}\label{eq:Inv-functorial-coalgebra-reindexing}
    \begin{gathered}
        \Sigma_f \colon \coAlg(G_X) \to \coAlg(G_Y) \\
        \Sigma_f(W,\beta) = (\Sigma_f W, \alpha_f\circ \Sigma_f\beta)
    \end{gathered}
\end{equation}
and hence by terminality of \(\Inv_Y\) a morphism of coalgebras:
\begin{equation}\label{eq:Inv-functorial}
  \Inv_f\colon\Sigma_f \Inv_X \to \Inv_Y
\end{equation}
sending an invertibility structure on some cell \(x\in X_{n+1}\) to an
invertibility structure on the cell \(f(x)\). It is easy to show that this
assignment is functorial by observing that for every pair of composable
morphisms \(f\colon X\to Y\) and \(g\colon Y\to Z\), the following
equations hold on the level of families and hence on the level of coalgebras as
well:
\begin{align}\label{eq:Inv-functorial-comp-equations}
    \begin{aligned}
        \Sigma_{\id} &= \id \\
        \alpha_{\id} &= \id
    \end{aligned} &&
    \begin{aligned}
        \Sigma_{g\circ f} &= \Sigma_g \circ \Sigma_f \\
        \alpha_{g\circ f} &= \alpha_g \Sigma_f \circ\ \Sigma_g \alpha_f
    \end{aligned}
\end{align}
By the universal property of \(\Inv_Z\), we conclude therefore that:
\begin{align}\label{eq:Inv-functorial-functoriality}
    \Inv_{\id} & = \id &
    \Inv_{gf} &= \Inv_g \circ\ \Sigma_g \Inv_f
\end{align}
Moreover, specializing to the case where \(f\colon C\to D\) is a morphism of
computads with invertible generators, we can check that \(\iSet(f)\) gives rise
to a morphism \(I_f\colon\Sigma_{\iKomp f} I_C \to I_D\) between the coalgebras
defined in the proof of Proposition~\ref{prop:inv-gens-are-inv}. By terminality,
we conclude that the following square commutes:
\begin{equation}\label{eq:Inv-of-iSet}
    \begin{tikzcd}
        \Sigma_{\iKomp f} I_C \arrow[r, "I_f"] \arrow[d, "\Sigma_f\inv_C"'] &
        I_D \arrow[d, "\inv_D"] \\
        \Sigma_{\iKomp f} \Inv_C \arrow[r, "\Inv_{\iKomp f}"] &
        \Inv_D
    \end{tikzcd}
\end{equation}

\begin{lemma}\label{lem:morph-of-computads-determined-gen-Inv}
    Let \(C\) be a computad with invertible generators and \(X\) an
    \(\omega\)\=/category. If a pair of morphisms \(f,g\colon \iKomp C\to X\)
    agree on generators of \(C\) and for every \(i\in \iSet_{n+1}(C)\), we have
    that
    \(
        \Inv_f(\inv_C(i)) = \Inv_g(\inv_C(i))
    \)
    then \(f = g\).
\end{lemma}
\begin{proof}
    We recall first that \(\iKomp C = \iKomp \Unf C\). The hypothesis of this
    lemma implies that the morphisms \(f\) and \(g\) agree on every generator
    of the ordinary computad \(\Unf C\). Being a morphism of
    \(\omega\)\=/categories, \(f\) must be given on cells of the form
    \(\coh_{B,A}[h]\) by the recursive formula:
    \begin{equation}\label{eq:morph-of-computads-determined-gen-Inv-coh}
        f(\coh_{B,A}[h]) =
            \gamma_X(\coh_{B,A}[\iCptd_{n+1+d}(f\circ h^\dagger)])
    \end{equation}
    where \(h^\dagger\colon \Pos B\to \Cell \Unf C\) is the transpose of \(h\)
    under the adjunction \(\Cptd\dashv \Cell\). The same formula holds for \(g\)
    as well, so by structural recursion, we conclude that \(f = g \).
\end{proof}

\begin{lemma}\label{lem:univ-prop-computads}
    Let \(C\) an \((n+1)\)\=/computad with invertible generators, and let
    \(X\) be an \(\omega\)\=/category. Morphisms \(f\colon \iKomp C\to X\)
    are in natural bijection with data of the form:
    \begin{itemize}
        \item a morphism \(f_n\colon \iKomp \sk_n \tr_n C\to X\)
        \item a function \(f_V\colon V_{n+1}^C \to X_{n+1}\)
        \item a function \(f_U\colon U_{n+1}^C \to \coprod_{x\in X_{n+1}}\Inv_{X,x}\)
    \end{itemize}
    subject to the source and target equations:
    \begin{equation}\label{eq:univ-prop-computads-source-target}
        \begin{aligned}
            \src f_V(v) &= f_n(\src \gen v) &
            \src {\pr_1}f_U(u) &= f_n(\src (\fwd (\igen u))) \\
            \tgt f_V(v) &= f_n(\tgt \gen v) &
            \tgt \pr_1f_U(u) &= f_n(\tgt (\fwd (\igen u)))
        \end{aligned}
    \end{equation}
\end{lemma}

\begin{proof}
    Given a morphism \(f\colon \iKomp C\to X\), we obtain such data by
    composing \(f\) with the counit \(\counitSk_{n,C}\) of the skeleton
    adjunction, looking at the image of the generators, and applying
    \(\Inv_f\) to the invertibility structures on the generators defined in
    Proposition~\ref{prop:inv-gens-are-inv}.

    Conversely, we will construct from the data \(f = (f_n,f_V,f_U)\)
    by induction on the dimension \(-1\le d \le\omega\) a morphism of globular
    sets and a function:
    \begin{align}\label{eq:univ-prop-computads-rho}
        \rho_{d,f}&\colon \Cell_{n+1+d} C\to \tr_{n+1+d}X &
        \rho'_{d,f}&\colon \iSet_{n+1+d}(C)\to \Inv_X
    \end{align}
    satisfying the following conditions:
    \begin{enumerate}
        \item \(\tr_{n+1+d'}\rho_{d,f} = \rho_{d',f}\) for every \(d'\le d\)
        \item the following square commutes where \(\gamma_X\colon TX\to X\) is
            the algebra structure of \(X\), \(T_{n+1+d}\) is the monad
            induced by the adjunction \(\Cptd_{n+1+d}\dashv \Cell_{n+1+d}\) and
            \(\counitCell_{n+1+d}\) is the counit of this adjunction:
            \begin{equation}\label{eq:univ-prop-computads-rho-comm-square}
                \begin{tikzcd}[column sep = 10em]
                    T_{n+1+d}\Cell_{n+1+d} C & \Cell_{n+1+d} C \\
                    T_{n+1+d}\tr_{n+1+d} X & \tr_{n+1+d} X
                    \arrow[from =1-1, to=1-2,
                        "\Cell_{n+1+d}\counitCell_{n+1+d}"]
                    \arrow[from =2-1, to=2-2, "\tr_{n+1+d}\gamma_X"']
                    \arrow[from= 1-1, to =2-1, "T_{n+1+d}\rho_{d,f}"']
                    \arrow[from= 1-2, to =2-2, "\rho_{d,f}"]
                \end{tikzcd}
            \end{equation}
        \item \(\rho'_{d,f}(i)\) is an invertibility structure on
            \(\rho_{d,f}(\fwd i)\) with inverses \(\rho_{d,f}(\linv i)\) and
            \(\rho_{d,f}(\rinv i)\) for every
            \(i\in \iSet_{n+1+d}(C)\)
    \end{enumerate}

    We define first \(\rho_{-1,f} = \tr_{n+1}f_n\) and
    \(\rho'_{-1,f}(i) = \Inv_{f_n}(\inv(i))\). This satisfies the second
    condition, being
    the truncation of a morphism of \(\omega\)\=/categories, and the other two
    by definition. Suppose now that \(d\in \Nat\). Then we define
    \(\rho'_{d,f}\) on cells of the form \(\lunit i\) and \(\runit i\)
    respectively by:
    \begin{equation}\label{eq:univ-prop-computads-rho-prime}
        \begin{aligned}
            \rho'_{d,f}(\lunit i) &= \ilunit(\rho'_{d-1,f}(i)) &
            \rho'_{d,f}(\runit i) &= \irunit(\rho'_{d-1,f}(i))
        \end{aligned}
    \end{equation}
    We define also when \(d=0\) the function \(\rho'_{0,f}\) on cells of the
    form \(\igen u\) to be given by \(f_U(u)\). We then define \(\rho_{d,f}\) on
    cells of dimension at most \(n+d\) to be given by \(\rho_{d-1,f}\) and on
    cells of dimension \(n+1+d\) by structural recursion. When \(d = 0\), for
    cells of the form \(\gen v\) for \(v\in V_{n+1+d}^C\), we define
    \(\rho_{0,f}\) to be given by \(f_V(v)\). For cells of the form \(\fwd i\),
    \(\linv i\) and \(\rinv i\), we define it to be \(\fwd\), \(\linv\) and
    \(\rinv\) of \(\rho'_{0,f}(i)\) respectively. This assignment commutes with
    the source and target morphisms by the
    equations~\eqref{eq:univ-prop-computads-source-target} and the third
    condition. Finally, on cells of the form \(\coh_{B,A}[g]\), we may assume
    recursively that \(\rho_{d,f}\) has already been defined on \(g(p)\) for
    every position \(p\in \Pos_{n+1+d}(B)\) and preserves sources and targets.
    We then define \(\rho_{d,f}\) to be:
    \begin{equation}\label{eq:univ-prop-computads-rho-coh}
        \rho_{d,f}(\coh_{B,A}[g]) =
            \gamma_X(\coh_{B,A}[\iCptd_{n+1+d}(\rho_{d,f} \circ g^\dagger)])
    \end{equation}
    where \(g^\dagger\colon \tr_{n+1+d}\Pos B\to \Cell_{n+1+d}C\) is the
    transpose of \(g\) under the adjunction \(\Cptd_{n+1+d}\dashv
    \Cell_{n+1+d}\). This assignment is compatible with the source
    functions by:
    \begin{equation}\label{eq:univ-prop-computads-rho-coh-bdry}
        \begin{aligned}
            \src (\rho_{d,f}&(\coh_{B,A}[g])) \\
                &= \gamma_X(T_{n+d}(\tr_{n+d}(\rho_{d,f}g^\dagger))(\src A)) \\
                &= \rho_{d-1,f}(\iCell_{n+d}(\counitCell_{n+d}\circ
                    \iCptd_{n+d}(\tr_{n+d} g^\dagger))(\src A)) \\
                &= \rho_{d-1,f}(\iCell_{n+d}(\tr_{n+d} g)(\src A)) \\
                &= \rho_{d-1,f}(\src(\coh_{B,A}[g]))
        \end{aligned}
    \end{equation}
    and the target functions similarly. This concludes the definition of
    \(\rho_{d,f}\). It satisfies the first and third condiiton by definition.
    The second one holds for generators of \(\Cptd_{n+1+d}\Cell_{n+1+d}C\) by
    the triangle equations of the adjunction. It holds for cells of the
    form \(\coh_{B,A}[h]\) by the following computation:
    \begin{equation}\label{eq:univ-prop-computads-rho-alg-mor}
        \begin{aligned}
          \rho_{d,f}&(\iCell_{n+1+d}(\counitCell_{n+1+d})(\coh_{B,A}[h])) \\
            &= \rho_{d,f}(\coh_{B,A}[\counitCell_{n+1+d}\circ h]) \\
            &= \gamma_X(\coh_{B,A}[\iCptd_{n+1+d}(\rho_{d,f}\circ
                (\counitCell_{n+1+d}\circ h)^\dagger)]) \\
            &= \gamma_X(\coh_{B,A}[\iCptd_{n+1+d}(\rho_{d,f})\circ\iCptd_{n+1+d}
                (\counitCell_{n+1+d})\circ h^\dagger)) \\
            &= \gamma_X(T_{n+d+1}(\rho'_{{d,f}})(\coh_{B,A}[h]))
        \end{aligned}
    \end{equation}
    Finally, we define \(\rho_{\omega,f}\) to be given on cells of dimension
    at most \(n+1+d\) by \(\rho_{d,f}\). This is well-defined by the first
    condition, and it is a morphism of \(\omega\)\=/categories \(\iKomp C\to X\)
    by the third condition. Moreover, it fully determines the family
    \(\rho'_{d,f}\) by the formula:
    \begin{equation}\label{eq:univ-prop-computads-inv-rho}
        \rho'_{d,f}(i) = \Inv_{\rho_{\omega,f}}(\inv(i))
    \end{equation}
    since both sides are morphisms of \(G_X\)\=/coalgebras into the terminal
    one \(\Inv_X\). The left hand side is a morphism of coalgebras by
    equations~\eqref{eq:univ-prop-computads-rho-prime}.

    Given a morphism \(h\colon \iKomp C\to X\) and letting \(g\) the data that
    it generates, we have that the morphisms \(h\) and \(\rho_{\omega,g}\) agree
    on every generator of \(C\) by definition of \(\rho_{0,g}\). Moreover, by
    induction on \(d\in \Nat\), we get that \(\rho'_{d,g}(i) = \Inv_h(\inv(i))\)
    using equations~\eqref{eq:univ-prop-computads-rho-prime}.
    Therefore, by Lemma~\ref{lem:morph-of-computads-determined-gen-Inv}, we
    conclude that \(h = \rho_{\omega,g}\).

    Conversely, given data \(g = (g_n, g_V, g_U)\) we can see that the data
    associated to \(\rho_{\omega,g}\) is precisely \(g\). For the second
    component, this is true by definition of \(\rho_{0,g}\). For the third
    component, we use equations~\eqref{eq:univ-prop-computads-inv-rho} and the
    definition of \(\rho'_{0,g}(\igen u)\) to conclude that they are equal.
    Finally, the morphisms \(\rho_{\omega,g}\circ \iKomp\counitSk_{n+1,C}\) and
    \(g_n\) agree
    on every generator by definition of \(\rho_{-1,g}\). They also agree on
    \(\iSet_{n}(\tr_nC)\) by definition of \(\rho'_{-1,g}\). Finally, by induction
    on \(d\in \Nat\), we can see that they also agree on
    \(\iSet_{n+1+d}(\tr_{n+d}C)\) using
    equations~\eqref{eq:univ-prop-computads-rho-prime}. Therefore, by
    Lemma~\ref{lem:morph-of-computads-determined-gen-Inv}, we conclude that
    \(\rho_{\omega,g}\circ \iKomp\counitSk_{n+1,C} = g_n\).
\end{proof}

\begin{proposition}\label{prop:inv-struct-repr}
    For every weak \(\omega\)\=/category \(X\) and every \(x\in X_{n+1}\), there
    exists a natural isomorphism:
    \begin{equation}\label{eq:inv-struct-repr}
        \Inv_{X,x} \cong
        \set{ g\colon\iKomp\Equiv^{n+1} \to X
            \such g(\fwd (\igen (c_{n+1}))) = x}
    \end{equation}
    where \(c_{n+1}\) is the top-dimensional generator of \(\Equiv^{n+1}\).
\end{proposition}
\begin{proof}
    To simplify the notation, let \(E_{X,x}\) be the family on the right-hand
    side and let \(c'_{n+1} = \igen (c_{n+1})\).
    Let \(\rho_X\colon E_X\to \Inv_X\) be given by the
    formula:
    \begin{equation}\label{eq:inv-struct-repr-univ-morphism}
        \rho_X(g) = \Inv_g(\inv(c'_{n+1}))
    \end{equation}
    Naturality of \(\rho\) follows then immediately
    from functoriality of \(\Inv\):
    \begin{equation}\label{eq:inv-struct-repr-univ-morphism-naturality}
        \rho_Y(h\circ g) = \Inv_h(\rho_X(g))
    \end{equation}
    Conversely, given a cell \(x\in X_{n+1}\) and an invertibility structure
    \(i\in \Inv_{X,x}\), we can construct a morphism
    \(F\Sphere^{n-1}\to X\) corresponding to the source and
    target of \(x\). Then Lemma~\ref{lem:univ-prop-computads} applies to give
    a morphism \(\rho^{-1}_{X}(i)\colon\iKomp\Equiv^{n+1}\to X\) sending
    \(\inv(c'_{n+1})\) to \(i\). In other words,
    \begin{equation}\label{eq:inv-struct-repr-section}
        \rho_X\circ \rho^{-1}_X = \id
    \end{equation}
    To complete the proof, it remains to check that
    \(\rho^{-1}_X(\rho_X(g)) = g\) for every morphism
    \(g\colon \iKomp\Equiv^{n+1}\to X\). The two morphisms agree on generators
    of dimension at most \(n\) by definition of \(\rho^{-1}_X\). They also send
    the unique element \(c'_{n+1}\) of \(\iSet_{n+1}(\Equiv^{n+1})\) to the same
    invertibility structure. Finally, they must also agree on every element of
    \(\iSet_{n+1+d}(\Equiv^{n+1})\) by
    equations~\eqref{eq:univ-prop-computads-rho-prime} and the fact that every
    element of \(\iSet_{n+1+d}(\Equiv^{n+1})\) is obtained by iterating the
    \(\lunit\) and \(\runit\) operations on \(c'_{n+1}\). Therefore, by
    Lemma~\ref{lem:morph-of-computads-determined-gen-Inv}, we conclude that
    \begin{equation}\label{eq:inv-struct-repr-retraction}
        \rho_X^{-1}\circ \rho_X = \id
    \end{equation}
    It follows that \(\rho^{-1}_X\) is inverse to \(\rho_X\).
\end{proof}

\begin{theorem}\label{thm:free-cats-as-colims}
  Free \(\omega\)\=/categories on computads with invertible generators can be
  expressed as the following colimits:
  \begin{enumerate}
    \item \(\iKomp\sk_{-1}C\) is initial for \(C\) the unique
        \((-1)\)\=/computad.
    \item \(\iKomp\sk_0 C\) is the copower \(V_0^C \cdot F\Disk^0\)
      for every \(0\)\=/computad \(C\).
    \item For every \((n+1)\)\=/computad \(C\), there exists a pushout square
      of the form:
      \begin{equation}\label{eq:free-cats-pushout}
        \begin{tikzcd}[column sep = 5em]
          V^C_{n+1}\cdot F\Sphere^{n}
            \amalg U^C_{n+1}\cdot F\Sphere^{n} &
          \iKomp \sk_n C_n \\
          V^C_{n+1}\cdot F\Disk^{n+1}
            \amalg U^C_{n+1}\cdot \iKomp \Equiv^{n+1} &
          \iKomp \sk_{n+1} C
          \arrow[from=1-1, to=1-2, "\phi_{n+1}^C \amalg \psi_{n+1}^C"]
          \arrow[from=1-1, to=2-1, hookrightarrow, swap,
            "V^C_{n+1}\cdot F\diskInc_{n+1} \amalg
            U^C_{n+1}\cdot \equivInc_{n+1}"]
          \arrow[from=1-2, to=2-2, "\iKomp\counitSk_n^{n+1}"]
          \arrow[from=2-1, to=2-2, "\gen \amalg\igen ", swap]
          \arrow[from=1-1, to=2-2, "\ulcorner"{very near end}, phantom]
        \end{tikzcd}
      \end{equation}
    \item For every computad \(C\), the \(\omega\)\=/category \(\iKomp C\) is
        the colimit of its skeleta:
        \begin{equation}\label{eq:free-cats-colim-skeleta}
          \begin{tikzcd}
            \iKomp\sk_{-1} C_{-1} \arrow[r, "\iKomp\counitSk_{-1}^{0}"] &
            \iKomp\sk_0 C_0 \arrow[r, "\iKomp\counitSk_0^{1}"] &
            \iKomp\sk_1 C_1 \arrow[r, "\iKomp\counitSk_1^{2}"] &
            \cdots \\
            &&& \iKomp C
            \arrow[from=1-1,to=2-4, "\iKomp\counitSk_{-1}"{swap}, bend right=15]
            \arrow[from=1-2,to=2-4, "\iKomp\counitSk_{0}"{swap}, bend right=10]
            \arrow[from=1-3,to=2-4, "\iKomp\counitSk_{1}"{swap}, bend right=5]
          \end{tikzcd}
        \end{equation}
    \end{enumerate}
\end{theorem}

\begin{proof}
    The first point is an immediate consequence of
    \(\iKomp\sk_{-1}C = F(\emptyset) = \emptyset\) having no cells. For the
    second point, we use that left adjoints commute with copowers, that
    \(\Disk^0\) is \(0\)\=/truncated and that every set is a copower of the
    singleton set, which is precisely \(\Cptd_0\Disk^0\), to compute that:
    \begin{equation}\label{eq:free-cats-as-colimits-zero}
        \begin{split}
            V_0^C \cdot F\Disk^0
                &\cong F(V_0^C \cdot \Disk^0)
                 \cong \iKomp\Cptd(V_0^C \cdot \Disk^0) \\
                &\cong \iKomp(V_0^C \cdot \Cptd\Disk^0)
                 \cong \iKomp(V_0^C \cdot \sk_0\Cptd_0\Disk^0) \\
                &\cong \iKomp\sk_0(V_0^C \cdot \Cptd_0\Disk^0)
                 \cong \iKomp\sk_0 C
        \end{split}
    \end{equation}
    The third point is an restatement of
    Lemma~\ref{lem:morph-of-computads-determined-gen-Inv}, using the
    representability result of Proposition~\ref{prop:inv-struct-repr}. Finally,
    for the last point, we use that the morphism \(\iCell(\counitSk_n^k)\) is
    bijective on cells of dimension at most \(n\), so the diagram above becomes
    a colimit of globular sets after applying the forgetful functor
    \(U\colon \Cat_\omega \to \gSet\). As the free \(\omega\)\=/category monad
    \(T\) preserves filtered colimits~\cite[Proposition~4.4]{
    deanComputadsWeak$omega$categories2024}, the forgetful functor \(U\)
    creates them, so the diagram above is also a colimit of
    \(\omega\)\=/categories.
\end{proof}



\section{The subcategory of rigid morphisms}
\label{subsec:rigid}

We conclude this article with a study of the subcategory of \emph{rigid}
morphisms between computads with invertible generators, often known as morphisms
of computads. These are the morphisms that preserve both kinds of generators.
In analogy with the case of ordinary computads, we show that rigid morphisms
form a wide subcategory of \(\iComp\) which is a presheaf topos.

\begin{definition}\label{def:rigid-morphisms}
  We define the class of \emph{rigid} morphisms recursively as follows:
  \begin{itemize}
    \item Every morphism of \(n\)\=/computads is rigid for \(n \le 0\).
    \item A morphism \(\sigma : C \to D\) of \((n+1)\)\=/computads for
      \(n\in \Nat\) is rigid when
      \begin{itemize}
        \item its truncation \(\sigma_{n} : C_{n} \to D_{n}\) is rigid,
        \item the cell \(\sigma_{V}(\gen(v))\) is of the form \(\gen(v')\) for
          every \(v \in V^{C}_{n}\),
        \item the element \(\sigma_{U}(\igen(u))\) is of the form \(\igen(u')\)
          for every \(u \in U^{C}_{n}\).
      \end{itemize}
    \item A morphism \(\sigma : C \to D\) of computads is rigid when
      \(\sigma_n\colon C_n\to D_n\) is rigid for every \(n \in \Nat\).
  \end{itemize}
\end{definition}

An easy consequence of the definition is that rigid morphisms are closed under
composition and contain identities. Therefore, computads and rigid
morphisms form a wide subcategory \(\iComprig_n\) of \(\iComp_n\) for every
\(n\in \Nat\cup\{-1,\omega\}\), whose inclusion we will denote by
\begin{equation}\label{eq:rigid-inclusion}
   \rigincl_n\colon \iComprig_n \to \iComp_n
\end{equation}
The truncation functors \(\tr^m_n\) for \(-1\le n < m \le \omega\) preserve
rigid morphisms by definition, so they restrict to functors between the
respective subcategories of rigid morphisms. Similarly, the skeleton
functors \(\sk^m_n\) restrict as well and so does the counit \(\counitSk^m_n\)
of the skeleton-truncation adjunction, since its components are rigid morphisms,
so the restriction of \(\sk^m_n\) is left adjoint to the restriction of
\(\tr^m_n\) with the same counit.

\begin{lemma}\label{lem:coskeleton}
  The truncation functors \(\tr^m_n\colon\iComprig_m\to\iComprig_n\) admit a
  right adjoint left inverse \(\cosk^m_n\) for all \(-1\le n < m \le \omega\).
\end{lemma}
\begin{proof}
  We will only construct a right adjoint to \(\tr^{n+1}_n\) for \(n\in \Nat\),
  since the general case can then be handled as in
  Proposition~\ref{prop:tr-is-rali}. We define the coskeleton functor
  \(\cosk^{n+1}_n\colon \iComp_n\to \iComp_{n+1}\) to send \(C\) to the
  \((n+1)\)\=/computad with underlying \(n\)\=/computad \(C\), sets of ordinary
  and invertible generators given by \(\Par_n\Cell_n C\), and attachment maps
  given by the identity function. Similarly, it sends a morphism \(f : C \to D\)
  to the morphism \(\cosk^{n+1}_n f\) consisting of \(f\) and the composites
  \(\gen_n\circ \Par_n\Cell_n f\) and \(\igen_n\circ \Par_n\Cell_n f\). We
  observe immediately that \(\cosk^{n+1}_n\) is left inverse to
  \(\tr^{n+1}_n\) and that for every \((n+1)\)\=/computad \(C\), we can define a
  rigid morphism:
  \begin{equation}\label{eq:coskeleton-counit}
    \unitCosk^{n+1}_{n,C}\colon C\to \cosk^{n+1}_n \tr^{n+1}_n C
  \end{equation}
  consisting of the identity of \(C_n\) and the functions
  \(\gen\circ\ \phi_{n+1}^C\) and \(\igen\circ\ \psi_{n+1}^C\). The morphisms
  \(\unitCosk^{n+1}_{n,C}\) are not in general natural in \(C\). However, they
  are natural with respect to rigid morphisms, and the coskeleton functor
  preserves rigid morphisms. It follows that \(\cosk^{n+1}_n\) restricts to a
  functor \(\iComprig_n\to \iComprig_{n+1}\) and \(\unitCosk^{n+1}_n\) defines
  a natural transformation \(\id\Rightarrow \cosk^{n+1}_n\tr^{n+1}_n\) between
  the restrictions. To show that \(\cosk^{n+1}_n\) is right adjoint to
  \(\tr^{n+1}_n\) with unit \(\unitCosk^{n+1}_n\) and counit the identity
  natural transformation, it remains to check the triangle identities:
  \begin{align}\label{eq:coskeleton-triangle}
    \tr^{n+1}_n \unitCosk^{n+1}_n &= \id &
    \unitCosk^{n+1}_n \cosk^{n+1}_n &= \id
  \end{align}
  which follow immediately from the definitions.
\end{proof}

\begin{corollary}\label{cor:icomprig-terminal}
  The category \(\iComprig_{n}\) has a terminal object \(\mathbbm{1}_n\).
\end{corollary}
\begin{proof}
  The coskeleton functor \(\cosk^n_{-1}\) is a right adjoint with source the
  terminal category, so it picks a terminal object of \(\iComprig_n\) for
  every \(n\).
\end{proof}

When restricting to the subcategory of rigid morphisms, the assignments
returning the sets of \(m\)\=/dimensional ordinary and invertible generators of
an \(n\)\=/computad for \(m \le n\) become functorial:
\begin{align}\label{eq:rigid-generator-assigment}
    V^{\bullet}_{m} &\colon \iComprig_n \to \Set &
    U^{\bullet}_{m} &\colon \iComprig_n \to \Set
\end{align}
Indeed, this follows by definition of rigid morphisms when \(m = n \in \Nat\),
and by composition with the truncation functors otherwise. Additionally, the
attachment maps of the generators define natural transformations of the form:
\begin{align}\label{eq:rigid-generator-attachment}
  \phi_{m}^{\bullet}
    &\colon V^{\bullet}_{m} \Rightarrow \Par_{m}\circ \Cell_n &
  \psi_{m}^{\bullet}
    &\colon U^{\bullet}_{m} \Rightarrow \Par_{m}\circ \Cell_n
\end{align}
Naturality of these assignments is a direct consequence of commutativity of the
squares~\ref{eq:def-morphism} in the definition of morphism of computads,
and the definition of the boundary natural transformation on generators and
invertible generators.

\begin{lemma}\label{eq:rigid-geometric-reflects-iso}
  A morphism of \(n\)\=/computads \(\sigma\colon C\to D\) is invertible if and
  only if it is rigid and induces bijections on the sets of ordinary and
  invertible generators in each dimension \(m \le n\). In particular, the
  functors \(V^{\bullet}_m\) and \(U^{\bullet}_m\) jointly reflect isomorphisms.
\end{lemma}
\begin{proof}
  Using that only generators can be sent to generators by morphisms of
  computads, we can conclude that if a composite \(g\circ f\) of morphisms is
  rigid, then the same must be true for \(f\). In particular, if
  \(f\colon C\to D\) is invertible, then \(f^{-1}\circ f = \id\) is rigid and
  hence so is \(f\). By functoriality of \(V^{\bullet}_m\) and
  \(U^{\bullet}_m\), the morphism \(f\) must induce bijections on the sets of
  generators. The converse is a straightforward induction on the dimension
  \(n\): the inverse of a rigid morphism that induces bijections on the sets of
  generators is obtained by the inverses of the bijections.
\end{proof}

Makkai's characterises presheaf topoi~\cite{makkaiWordProblemComputads} as the
cocomplete categories that can be equipped with a functor to \(\Set\) that is
conservative, cocontinuous, and whose category of elements is a small disjoint
union of full subcategories each with an initial object. The last condition is
equivalent to the functor being familially representable, as the category of
elements functor preserves and reflects coproducts, and a functor to \(\Set\) is
representable if and only if its category of elements has an initial object.

We will apply Makkai's characterisation to show that the category
\(\iComprig_n\) is a presheaf topos for every \(n\in \Nat\cup\{-1,\omega\}\) by
considering the functor:
\begin{align}\label{eq:rigid-geometric-functor}
    \abs{-} &\colon \iComprig_n \to \Set &
    \abs{C} &= \coprod_{m \le n} (V^C_m \amalg U^C_m)
\end{align}
sending a computad to the disjoint union of its sets of generators. This functor
is conservative by Lemma~\ref{eq:rigid-geometric-reflects-iso}. We then show
that is \(\iComprig_n\) is cocomplete and that \(\abs{-}\) is cocontinuous. The
proof is completely analogous to the one given by Dean et
al.~\cite{deanComputadsWeak$omega$categories2024} for ordinary computads.

\begin{proposition}\label{prop:icomprig-cocomplete}
  The category \(\iComprig_{n}\) is cocomplete, and the functor \(\rigincl_{n}\)
  preserves colimits for every \(n\in \Nat\cup\{-1,\omega\}\).
\end{proposition}
\begin{proof}
  The base cases \(n = -1\) and \(n = 0\) are trivial, since \(\rigincl_{0}\) is
  the identity functor on the terminal category or the category of sets
  respectively, which are cocomplete. Suppose now that \(\iComprig_n\) is
  cocomplete for some \(n\in \Nat\) and that \(\rigincl_n\) is cocontinuous.
  Given a small diagram \(J : \mathcal{D} \to \iComprig_{n+1}\), we may then
  form the following colimits:
  \begin{equation}\label{eq:icomprig-cocomplete-definition}
    \begin{aligned}
      C_n &= \colim (\tr^{n+1}_n \circ\ J) \\
      V^C_{n+1} &= \colim (V^{\bullet}_{n+1} \circ J) \\
      U^C_{n+1} &= \colim (U^{\bullet}_{n+1}\circ J)
    \end{aligned}
  \end{equation}
  The whiskerings by \(J\) of the natural transformations
  \(\phi^{\bullet}_{n+1}\) and \(\psi^{\bullet}_{n+1}\) induce universal
  morphisms between the colimits:
  \begin{equation}\label{eq:icomprig-cocomplete-definition-att-1}
    \begin{aligned}
      \colim(\phi^{\bullet}_{n+1} J)
        &\colon V_{n+1}^C \to
          \colim(\Par_{n}\circ\Cell_{n}\circ\tr^{n+1}_n \circ\ J) \\
      \colim(\psi^{\bullet}_{n+1} J)
        &\colon U_{n+1}^C \to
          \colim(\Par_{n}\circ\Cell_{n}\circ\tr^{n+1}_n \circ\ J)
    \end{aligned}
  \end{equation}
  Composing these morphisms with the canonical morphism:
  \begin{equation}
    \label{eq:icomprig-cocomplete-definition-att-2}
    \colim(\Par_{n}\circ\Cell_{n}\circ\tr^{n+1}_n \circ\ J)
      \to \Par_{n}\circ\Cell_{n}(C_n)
  \end{equation}
  induced by the universal property of the colimit and the cocone defining
  \(C_n\), we obtain an \((n+1)\)\=/computad \(C\). The cocones defining
  \(C_n\), \(V^C_{n+1}\) and \(U^C_{n+1}\) can be combinined to define a cocone
  for \(J\) with apex \(C\) consisting of rigid morphisms. The universal
  property of this cocone both in \(\iComprig_{n+1}\) and in \(\iComp_{n+1}\)
  is an immediate consequence of the universal properties of \(C_n\) and the
  colimits defining \(V^C_{n+1}\) and \(U^C_{n+1}\). Finally, the case
  \(n = \omega\) follows by strict commutativity of the construction of the
  colimit above with the truncation functors.
\end{proof}

\begin{corollary}\label{prop:icomprig-cocontinuous}
  The functors \(V^\bullet_m\) and \(U^\bullet_m\) are cocontinuous and so is
  \(\abs{-}\).
\end{corollary}
\begin{proof}
  This is immediate from the constuction of colimits when \(n = m\in \Nat\).
  The case \(m < n\) follows by the previous case and by the truncation functors
  being left adjoints and hence cocontinuous. Finally, cocontinuity of
  \(\abs{-}\) follows from that of \(V^\bullet_m\) and \(U^\bullet_m\) and
  colimits commuting with colimits.
\end{proof}

To complete the proof that \(\iComprig_n\) is a presheaf topos, it remains to
show that the functor \(\abs{-}\) is familially representable. To do that, we
will consider first the restriction of the functor of cells to the subcategory
of rigid morphisms:
\begin{equation}\label{eq:rigid-cell-functor}
    \iCellrig_n = \iCell_n \circ\ \rigincl_n \colon \iComprig_n \to \gSet_n \\
\end{equation}
We will show that the functor \(\iCellrig_n\) is familially representable and
then deduce a representation of \(\abs{-}\) from that.

\begin{theorem}\label{prop:cellrig-repr}
  The functor \(\iCellrig_n\) is familially representable for
  \(n\in\Nat\cup\set{\omega}\).
\end{theorem}
\begin{proof}
  We will show by induction on \(n\in\Nat\) that the functors
  \(\iCellrig_n\), \(\Par_n\iCellrig_n\) and
  \(\iSetrig_n = \iSet_n\circ\ \rigincl_n\) are familially representable. To do
  that consider the \(n\)\=/globular set \(P_n = \iCellrig_n(\mathbbm{1}_n)\)
  and the sets \(Q_n = \Par_n P_n\) and \(L_n = \iSetrig_n(\mathbbm{1}_n)\).
  We can define then a family of functors indexed by \(m\le n\) and \(p\in P_m\)
  by:
  \begin{equation}\label{eq:cellrig-repr-functors}
    \begin{gathered}
      \iCellrig_n(-;p) \colon \iComprig_n \to \Set \\
      C \mapsto \set{c \in \Cell_n(C)_m \such \iCellrig_n(!_C)(c) = p}
    \end{gathered}
  \end{equation}
  and we can define functors \(\Par_n\iCellrig_n(-;q)\) for
  \(q\in Q_n\) and \(\iSetrig_n(-;l)\) for \(l\in L_n\) completely analogously.
  These functors give a decomposition of \(\iCellrig_n\):
  \begin{equation}\label{eq:cellrig-repr-decomposition}
     \begin{gathered}
        \iCellrig_n(-;p) \colon \iComprig_n \to \Set \\
        C \mapsto \set{c \in \Cell_n(C)_m \such \iCellrig_n(!_C)(c) = p}
     \end{gathered}
  \end{equation}
  and similarly of \(\Par_n\iCellrig_n\) and \(\iSetrig_n\). Therefore, it
  suffices to show that each of them are representable.

  We start with the base case \(n = 0\), where \(\iComprig_0 = \Set\). In this
  case, \(\iSetrig_{0}\) is constant at the empty set, so familially
  representable by the empty family. The functor \(\iCellrig_0\) is the identity
  functor on \(\Set\), so representable by the terminal set \(\mathbbm{1}_0\).
  The functor \(\Par_0\iCellrig_0\) is the diagonal functor, sending a set
  \(X\) to \(X\times X\) so it is represented by the coproduct
  \(\mathbbm{1}_0\amalg\mathbbm{1}_0\).
  For the inductive step, suppose that we have constructed representations for
  some \(n\in \Nat\) in order to construct them for \(n+1\).

  We will first
  start by constructing representations for \(\iSetrig_{n+1}(-;l)\) for
  \(l\in L_{n+1}\). For that, assume  that \(l = \igen q\) for some
  invertible generator \(q\) of \(\mathbbm{1}_{n+1}\). By the construction of
  the terminal object, we have that \(q \in Q_n\), so we may assume that
  \(\Par_n\iCellrig_n(-;q)\) is represented by some \(n\)\=/computad \(C_q\)
  with universal element \(A_q\in \Par_n\iCellrig_n(C_q;q)\). We define then
  \(C_l\) the be the \((n+1)\)\=/computad consisting of \(C_q\) and a unique
  invertible generator \(u\) with \(\psi^C_{n+1}(u) = A_q\). By definition of
  morphisms and \(C_q\), we see that morphisms \(f\colon C_l\to C\) correspond
  to pairs of \(A'\in \iCellrig_n(C;q)\) and \(u'\in U^C_{n+1}\) such that
  \(\psi^C_{n+1}(u') = A'\). These are in bijection to generators
  \(u'\in U^C_{n+1}\) such that
  \(\Par_n\iCellrig_n(!_C)(\psi^C_{n+1}(u')) = q\). The last equation is
  equivalent to \(\iSetrig_{n+1}(!_C)(\igen u') = l\), so get that \(C_l\)
  represents \(\iSetrig_{n+1}(-;l)\) with universal element \(\igen u\).

  Suppose then that \(l = \lunit l'\) for some \(l'\in L_{n}\). By the inductive
  hypothesis, we have that \(\iSetrig_n(-;l')\) is represented by some
  \(n\)\=/computad \(C_{l'}\) with universal element
  \(i_{l'}\in \iSetrig_n(C_{l'};l')\). We can then easily deduce that
  \(\iSetrig_n(-;l')\) is represented by the \((n+1)\)\=/computad
  \(C_l = \sk^{n+1}_n C_{l'}\) with universal element \(\lunit l'\). The case
  where \(l = \runit l'\) is completely analogous.

  We proceed now to show that \(\iCellrig_{n+1}(-;p)\) is representable by
  structural recursion on the cell \(p\in P_{n+1}\). The case where
  \(p = \gen q\) is analogous to the one for invertible generators: we let
  \(C_p\) be the \((n+1)\)\=/computad obtained by adding a unique ordinary
  generator \(v\) to the \(n\)\=/computad \(C_q\) with \(\phi^C_{n+1}(v) = A_q\)
  and we let the universal element be \(\gen v\). For the cases where
  \(p = \fwd l\), \(p = \linv l\) or \(p = \rinv l\), we may assume that
  \(\iSetrig_n(-;l)\) is represented by some computad \(C_l\) with
  universal element \(i_l\in \iSetrig_n(C_l;l)\) and let \(C_p = C_l\) with
  universal element \(\fwd i_l\), \(\linv i_l\) or \(\rinv i_l\) respectively.

  Suppose therefore that \(p = \coh_{B,A}[f]\) is a coherence cell for some
  morphism \(f\colon \Cptd_{n+1}(\Pos B)\to \mathbbm{1}_{n+1}\), and let
  \(f^\dagger\colon \Pos B \to P_{n+1}\) its transpose. By
  structural recursion, we may assume that a computad \(C_{f^\dagger(q)}\) has
  been defined for every \(q\in \Pos_m B\) representing
  \(\iCellrig_n(-;f^\dagger(q))\). We can then define \(C_p\) to be the colimit:
  \begin{equation}
    C_p = \colim_{q\in \Pos_m B} C_{f^\dagger(q)}
  \end{equation}
  indexed by the category of elements of the globular set \(\Pos B\) with
  morphisms given by the universal morphisms corresponding to source and target
  natural transformations \(\iCellrig_{n+1}(C)_{m+1}\Rightarrow
  \iCellrig_{n+1}(C)_m\). Morphisms \(g\colon C_p\to C\) correspond to cells
  \(g^\dagger(q)\in \iCellrig_n(C;f^\dagger(q))\) subject to source and boundary
  conditions. More precisely, the source and target conditions amount to
  \(g^{\dagger}\) being a morphisms of globular sets
  \(\Pos B \to \iCell_n(C)\) such that \(\iCell_{n+1}(!_C)\circ g^\dagger =
  f^\dagger\). Transposing \(g^\dagger\), we get a morphism of computads
  \(g'\colon \Cptd_{n+1}(\Pos B)\to C\) such that \(!_C\circ g' = f\) and
  hence a cell \(c = \coh_{B,A}[g']\in \iCellrig_{n+1}(C;p)\).

  Finally, suppose that some \(q = (p, p')\in Q_{n+1}\) is given and let
  \(q' = \bdry p = \bdry p' \in Q_n\). By the inductive hypothesis, we may
  assume that \((n+1)\)\=/computads \(C_{p}\) and \(C_{p'}\) representing
  \(\iCellrig_n(-;p)\) and \(\iCellrig_n(-;p')\) have been given, as well as
  an \(n\)\=/computad \(C_{q'}\) representing \(\Par_n\iCellrig_n(-;q')\). It
  follows easily that the \((n+1)\)\=/computad \(C_q\) obtained as the pushout:
  \begin{equation}\label{eq:cellrig-repr-spheres}
    \begin{tikzcd}
      \sk^{n+1}_nC_{q'} \ar[r]\ar[d]
      \ar[dr, "\ulcorner"{very near end}, phantom] & C_{p}
      \ar[d, dashed] \\
      C_{p'} \ar[r, dashed] & C_{q}
    \end{tikzcd}
  \end{equation}
  represents \(\iCellrig_{n+1}(-;q)\). The solid morphisms in the diagram
  are the universal morphisms corresponding to the boundary natural
  transformation.

  This completes the induction on \(n\in\Nat\). The case \(n = \omega\) follows
  immediately from the finite case and the skeleton and truncation adjunction.
\end{proof}

\begin{corollary}\label{cor:abs-repr}
  The functor \(\abs{-}\) is familially representable.
\end{corollary}
\begin{proof}
  Using the notation of the previous proof, for all finite \(m\le n\),
  we have that the
  functors \(V^\bullet_m\) and \(U^\bullet_m\) are familially representable as
  coproducts the representable functors:
  \begin{equation}\label{eq:abs-repr}
    \begin{aligned}
      V_m^\bullet
        &= \coprod_{p\in V_m^{\mathbbm{1}_n}} \iCellrig_n(-; \gen p) \\
      U_m^\bullet
        &= \coprod_{p\in U_m^{\mathbbm{1}_n}} \iCellrig_n(-;\fwd \igen p)
    \end{aligned}
  \end{equation}
  Since familially representable functors are closed under coproducts, the
  functor \(\abs{-}\) must also be familially representable.
\end{proof}

\begin{corollary}\label{cor:icomprig-topos}
  The category \(\iComprig_{n}\) is a presheaf topos for
  \(n\in\Nat\cup\set{-1,\infty}\).
\end{corollary}
\begin{proof}
  We apply Makkai's characterisation of presheaf topoi to the functor
  \(\abs{-}\): the source category \(\iComprig_n\) is cocomplete by
  Proposition~\ref{prop:icomprig-cocomplete}, and the functor is conservative by
  Lemma~\ref{eq:rigid-geometric-reflects-iso}, cocontinuous by
  Corollary~\ref{prop:icomprig-cocontinuous} and familially representable by
  Corollary~\ref{cor:abs-repr}. More explicitly, Makkai's characterisation
  implies that \(\iComprig_n\) is equivalent to the category of presheaves on
  the category \(\mathcal{V}_n\) with objects the elements of
  \(\abs{\mathbbm{1}_n}\) and with morphisms \(p\to p'\) the morphisms between
  the computads representing the functors \(\iCellrig_n(-;p)\) and
  \(\iCellrig_n(-;p')\). The equivalence is the nerve-realisation
  adjunction induced by the inclusion of \(\mathcal{V}_n\) into \(\iComprig_n\),
  sending \(p\) into the computad representing \(\iCellrig_n(-;p)\).
\end{proof}


\section*{Acknowledgements}

We would like to thank Jamie Vicary for his support on this project. We would
also like to help Noam Zeilberger, Paul-André Melliès, and François Métayer at
IRIF.

\bibliographystyle{elsarticle-num}
\bibliography{bibliography}

\end{document}